\documentclass[12pt]{article}
\usepackage{a4}
\usepackage{amsthm}
\usepackage{amsfonts}
\usepackage{amssymb}
\usepackage{amsmath} 
\usepackage{array}
\usepackage[utf8]{inputenc}
\usepackage{graphicx}
\usepackage{cite}
\usepackage{color}
\usepackage{hyperref}
\usepackage{enumitem}
\usepackage{comment} 
\usepackage{algpseudocode}
\usepackage{algorithm}
\usepackage{setspace}
\usepackage{paths}

\title{Edge-partitioning 3-edge-connected graphs into paths~\thanks{Both authors were partially supported by ANR project Stint under reference ANR-13-BS02-0007 and by the LABEX MILYON (ANR-10-LABX-0070) of Université de Lyon, within the program Investissements d'Avenir (ANR-11-IDEX-0007) operated by the French National Research Agency (ANR). Klimošová was also supported by Center of Excellence – ITI, project P202/12/G061 of GA ČR and by Center for Foundations of Modern Computer Science (Charles Univ. project UNCE/SCI/004).}~\thanks{Extended abstract of this work was published as: T. Klimo\v{s}ov\'a, S. Thomass\'e: Decomposing graphs into paths and trees, Electron. Notes Discrete Math., 61 (2017) 751-757.}}
\author{Tereza Klimo\v{s}ov\'a\thanks{Department of Applied Mathematics, Faculty of Mathematics and Physics, Charles University,  
Malostransk\'e n\'am\v{e}st\'i 25, 118 00 Praha 1, Czech Republic. E-mail: {\tt tereza@kam.mff.cuni.cz}.} \and St\'ephan Thomass\'e\thanks{Laboratoire d’Informatique du Parall\'elisme,
\'Ecole Normale Sup\'erieure de Lyon,
69364 Lyon Cedex 07, France. E-mail: {\tt stephan.thomasse@ens-lyon.fr}.}~\thanks{Institut Universitaire de France}}

\begin{document}
\maketitle
\begin{abstract}
We show that for every $\ell$, there exists $d_\ell$ such that every $3$-edge-connected graph with minimum degree $d_\ell$ can be edge-partitioned into paths of length $\ell$ (provided that its number of edges is divisible by $\ell$). This improves a result asserting that 24-edge-connectivity and high minimum degree provides such a partition. This is best possible as 3-edge-connectivity cannot be replaced by 2-edge connectivity. 
\end{abstract}

\section{Introduction}\label{sec:intro}
Given a graph $G$, we denote by $V(G)$ and by $E(G)$ its vertex set and its edge set, respectively. For $X\subseteq V(G)$, $G[X]$ denotes the induced subgraph of $G$ on $X$. Unless we specify otherwise, we consider graphs to be simple graphs without loops, and multigraphs to have multiple edges and loops.

For graphs $G$ and $H$, we say that $G$ is {\em $H$-decomposable} if there exists a partition $\{E_i\}_{i\in[k]}$ of $E(G)$ such that every $E_i$ forms an isomorphic copy of $H$. We then call $\{E_i\}_{i\in[k]}$ an {\em $H$-decomposition} of $G$. Note that if $G$ has an $H$-decomposition, $|E(H)|$ divides $|E(G)|$.

In~\cite{bib:barat-thomassen}, Bar\'at and Thomassen conjectured that for a fixed tree $T$, every sufficiently edge-connected graph with number of edges divisible by $|E(T)|$ has a $T$-decomposition. Recently, the conjecture was proved by Bensmail, Harutyunyan, Le, Merker and the second author~\cite{bib:BHLMT}.

\begin{theorem}\label{thm:barat-thomassen}
For any tree $T$, there exists an integer $k_T$ such that every $k_T$-edge-connected graph with number of edges divisible by $|E(T)|$ has a $T$-decomposition.
\end{theorem}

The starting point of the Bar\'at and Thomassen conjecture is the following weakening of the 3-flow conjecture of Tutte: every highly enough edge-connected graph has a nowhere zero 3-flow. This statement turned out to be equivalent to the Bar\'at-Thomassen conjecture restricted to the claw (star with three leaves), and was first proved for edge-connectivity 8 by Thomassen in~\cite{bib:thomassen-3-flow} and then improved by Lov\'asz, Thomassen, Wu and Zhang in~\cite{bib:LOVASZ2013587}. One of the important motivation of the Bar\'at-Thomassen conjecture was to provide tools which would decompose highly edge-connected graph into bounded size trees. Ideally, one would have expected a very general decomposition result giving rise to absorber techniques, as in the regularity lemma. However, the proof of the Bar\'at-Thomassen conjecture used as a starting point the claw case which is based on a very tight induction, leaving as open a more general sparse graph decomposition result.

A natural question which could lead to a better understanding of sparse graphs  is to ask how much edge-connectivity is important compared to minimum degree. In~\cite{bib:orig-paths}, the authors posed the following strengthened version of the conjecture of Bar\'at and Thomassen and they proved it for $T$ being a path.

\begin{conjecture}\label{conj:wrong}
There is a function $f$ such that, for any fixed tree $T$ with
maximum degree $\Delta_T$, every $f(\Delta_T)$-edge-connected graph with  minimum degree at least $f(|E(T)|)$ and number of edges divisible by $|E(T)|$ has a $T$-decomposition.
\end{conjecture}

We define the {\em length} of a path as its number of edges. The theorem for paths can be then stated as:

\begin{theorem}\label{thm:orig-paths}
For every integer $\ell$, there exists $d=d(\ell)$ such that the edge set of every $24$-edge-connected graph $G$ with minimum degree $d$ and number of edges divisible by $\ell$ has a decomposition into paths of length $\ell$.
\end{theorem}

Our main result, stated in Section~\ref{sec:paths}, is that $24$ in the statement of Theorem~\ref{thm:orig-paths} can be replaced by $3$. 

\begin{theorem}\label{thm:main}
For every integer $\ell$, there exists $d=d(\ell)$ such that the edge set of every $3$-edge-connected graph $G$ with minimum degree $d$ and number of edges divisible by $\ell$ has a decomposition into paths of length $\ell$.
\end{theorem}

The dependency on connectivity is optimal. In~\cite{bib:orig-paths}, it was shown that there exists $2$-edge-connected graphs with arbitrarily high minimum degree and number of edges divisible by $9$ which do not have a decomposition into paths of length $9$. 

However, Conjecture~\ref{conj:wrong} is not true in general. In ~\cite{bib:our-trees}, we showed that it does not hold even for trees of maximum degree three. 

Whereas in our result we made effort to prove the optimal connectivity bound for decomposition into paths, we believe that there should be a substantially simpler way of proving non-optimal connectivity bounds. To motivate further research in this area, we pose the following alternative to Conjecture~\ref{conj:wrong}, which bounds connectivity in terms of the number of leaves of the tree.

\begin{conjecture}\label{conj:new}
There is a function $f$ such that, for any fixed tree $T$ with $m$ leaves, every $f(m)$-edge-connected graph with minimum degree at least $f(|E(T)|)$ and number of edges divisible by $|E(T)|$ has a $T$-decomposition.
\end{conjecture}

Observe that our result implies that we can set $f(2)=3$ and that the existence of $f(3)$ contains the claw case. It would be very interesting to find a proof which would not use the claw case as a starting point.

The core idea of the proof of Theorem~\ref{thm:main} is to iterate partitions $(A,B)$ of the input graph $G$ along edge-cuts of bounded size in such a way that $A$ is highly connected. This can be easily done by the following result.

\begin{lemma}\label{lem:cut}
Let $k$ be an integer and $G$ be a multigraph. Then, there is a cut $(A,B)$ in $G$ of order at most $2k$ such that $G[A]$ is $k$-edge-connected or has only one vertex.
\end{lemma}

The trick is to observe that the connectivity of $A$ gives the opportunity to extend any (recursively found) decomposition of $B$, i.e. $A$ plays the role of an absorber. The first difficulty is to be able to keep the information of the edges of the cut $(A,B)$. This is achieved by contracting $A$ to a single vertex $a$ and using Mader's splitting theorem to "push" the edges incident to $a$ into $B$, creating new edges in $B$ called "wiggly". However, this tools works very well on even degree vertices but does not apply on, say, degree 3 vertices created by contraction of a 3-edge cut. This is the main problem of this technique (and maybe the main interest of the paper). To overcome this, we allow hyperedges of size 3 which are created by splitting degree 3 vertices $a$. Fortunately the notion of partition-connectivity of hypergraphs fits very well to our goal. Moreover, Lemma~\ref{lem:cut} admits a straightforward hypergraph generalization.

The second difficulty (and the paramount one for trees) is that we need to remember in $B$ the number of edges inside $A$ modulo $\ell$. This can easily be done by assigning some length to the created wiggly edges, hence a wiggly edge $e$ can now be interpreted as a potential path of a given length which can be realized in $A$. We then index the two endvertices of $e$ by the label $A$, and keep then the information of the origin of $e$.

The third difficulty is that if several wiggly edges indexed by $A$ are used in a path-decomposition of $B$, their realization in $A$ may intersect. To overcome this, we need to constrain the use of same-index edges in our path decomposition.

To sum-up, we are dealing here with {\em complex hypergraphs}, which can have half-edges (called stubs), loops, edges, and hyperedges of size three. Each edge has a label corresponding to its length and every incidence vertex/edge has an index. Somewhat unexpectedly this definition is robust enough to allow induction, but the price to pay is that even the definition of path in a complex hypergraph is not straightforward. 

Ideally, the presentation of the proof would start by the general picture, and then introduce the tools and key definitions. However, the general idea is so intricately linked to the notion of complex hypergraph that we chose to start with the basics and postpone the final proof to the end of the paper.

We define complex hypergraphs in Section~\ref{sec:complex}.
Before constructing path-decompositions, we need to "shrink" the hyperedges of size three into edges without decreasing the degree of any vertex too much. The existence of such a shrinking is shown by an entropy compression argument in Section~\ref{sec:shrink}. We then construct partial path-decompositions using the notion of path-graphs developed in~\cite{bib:orig-paths}. We introduce this notion in Section~\ref{sec:path-graphs} and we prove the key lemma regarding the existence of a decomposition in Section~\ref{sec:core}. Finally, in Section~\ref{sec:paths}, we prove Theorem~\ref{thm:main}, showing how to iteratively construct a path-decomposition of the whole graph by finding a path-decomposition of the last (complex) graph in the sequence and then adding the remaining parts one by one, repeatedly routing the paths across the cut to the next part and extending the decomposition to the next part.

Throughout Sections~\ref{sec:complex},~\ref{sec:tools},~\ref{sec:shrink},~\ref{sec:path-graphs}, and~\ref{sec:core}, where we introduce terminology and tools needed for proving Theorem~\ref{thm:main}, we assume that $\ell$ is a fixed integer greater than one.







\section{Complex hypergraphs}\label{sec:complex}

A {\em hypergraph} $H$ is a pair $(V,E)$ where $V$ is a set of vertices and $E$ is a multiset of {\em hyperedges} where each hyperedge is a nonempty multiset of elements of $V$. For instance, $[u,v,v]$ and $[v,v,v]$, where $u,v\in V$ are hyperedges of size three. As in the case of graphs, $H[X]$ denotes the {\em induced hypergraph} of $H$ on the vertex set $X\subseteq V$ and its hyperedges are those hyperedges of $H$ which contain only vertices in $X$. 
A {\em path} in a hypergraph $H$ is an alternating sequence, without repetition, of vertices
and hyperedges, $v_1, e_1, v_2, e_2,\ldots,e_k, v_{k+1}$ such that $v_i,v_{i+1}\in e_i$ for every $i\in [k]$.
To {\em shrink a hyperedge $e$} means to replace it by $e'\subsetneq e$.
A hypergraph is {\em connected} if there is a path between any two vertices.
A hypergraph is {\em $k$-edge-connected} if after removing any $k-1$ hyperedges the hypergraph is still connected. 

Now we define {\em complex graphs and hypergraphs} which are the central objects of this proof. Since we will repeatedly perform some cuts, some edges will be divided into two half edges, called here stubs. To keep track of these, and remember the information of their creation, we need higher order relations, like wiggly edges and hyperedges of size $3$ which, in a sense, represent a union of two or three stubs. They allow more control over divisibility and connectivity of the hypergraph we construct.

We define a {\em complex hypergraph $G$} to be a tuple $(V,E_o,E_w,H,S)$, where 
\begin{itemize}
\item $(V,E_o)$ is a simple graph, 
\item $E_w$ is a multiset of {\em wiggly edges} $e=([(v_1,i_1^e),(v_2,i_2^e)],\alpha)$, where $v_1,v_2\in V$ (not necessarily distinct), $i_1^e,i_2^e\in \ZZ$ and  $\alpha\in \NNo$, 
\item $H$ is a multiset of hyperedges $h=([(v_1,i_1^h),(v_2,i_2^h),(v_3,i_3^h)],\alpha)$, where $v_1,v_2,v_3\in V$ (not necessarily distinct), $i_1^h,i_2^h,i_3^h\in \ZZ$ and $\alpha\in \NNo$,  
\item $S$ is a multiset of {\em stubs} $s=([(v_1, i_1^s)], \alpha)$, where $v\in V$, $i_1^s\in \ZZ$ and $\alpha\in \NNo$. 
\end{itemize}

In other words, hyperedges, wiggly edges and stubs are ordered pairs, such that the first entry is a multiset $V_e$ of pairs $(v,i)$ where $v$ is a vertex and $i\in \ZZ$, of size $3,2$ and $1$, respectively, and the second entry is a number $\alpha\in \NNo$. We say that a vertex $v$ and a hyperedge, wiggly edge or stub $e$ are {\em incident} if $(v,i)\in V_e$ for some $i\in \ZZ$. We then call $i$ the {\em index} of $e$ at $v$. Note that if $e$ is a wiggly edge or a hyperedge, it might have several different indices at $v$.
For $v\in V$ and $i\in \mathbb{N}$, we define $\iota_G(v,i)$ to be the number of occurrences of the pair $(v,i)$ in stubs, wiggly edges and hyperedges incident with $v$ in $G$, $\iota(G)$ is then defined as the maximal $\iota_G(v,i)$, that is $\iota(G)=\max_{v\in V, i\in \mathbb{N}} \iota_G(v,i)$. We emphasize that we do not consider index $-1$ when computing $\iota_G(v,i)$ and $\iota(G)$.

We call $\alpha=\alpha(e)$ the {\em length} of a wiggly edge, hyperedge or stub $e$, respectively. For convenience, we say that ordinary edges have length one, writing $\alpha(e)=1$ for an ordinary edge $e$. 
At some occasions, we do not specify index or length since they are irrelevant for our purposes. The sum of the lengths of all the stubs, edges and hyperedges in a complex hypergraph $G$ is the {\em length of $G$} and is denoted by $\alpha(G)$. A complex hypergraph is {\em $\ell$-divisible} if its length is divisible by $\ell$.

If $H=\emptyset$, we say that $G$ is a {\em complex graph}.
We refer to the set $E_o$ as {\em ordinary edges} and to the set $E_o\cup E_w$ as {\em edges}. We denote the set of edges of a complex hypergraphs $G$ by $E(G)$. We emphasize that, unlike in the case of "usual" hypergraphs, in the context of complex hypergraphs, we reserve the term {\em hyperedge} only for elements of $H$. 

The degree $\deg_G v$ of a vertex $v$ in a complex hypergraph $G$ is the number of {\em occurrences of $v$} in stubs, edges and hyperedges. That is, the number of times a pair $(v,i)$ , $i\in \NNo$, appears in stubs, wiggly edges and hyperedges incident with $v$ plus the number of times $v$ appears in ordinary edges. E.g., in a hyperedge $h=([(u,i_1),(v,i_2),(v,i_3)],\alpha)$, $u$ occurs once and $v$ occurs twice, therefore, $h$ contributes one to the degree of $u$ and two to the degree of $v$.
The {\em unit degree} $\udeg_G v$ of a vertex $v$ is the number of ordinary edges plus the number of occurrences of $v$ in wiggly edges of length $1$. The {\em stub-degree} $\sdeg_G v$ of a vertex $v$ is the number of stubs incident with $v$, the {\em edge-degree} $\edeg_G v$ is the number of ordinary edges incident with $v$ plus the number of occurrences of $v$ in wiggly edges and the {\em hyper-degree} $\hdeg_G v$ is the number occurrences of $v$ in hyperedges. A {\em loop} is a wiggly edge of the form $([(v,i),(v,j)],\alpha)$, otherwise we call an edge a {\em non-loop}, in particular, all ordinary edges are non-loops. We define  the {\em loop degree $\ldeg_G v$ of $v$} to be twice the number of loops incident with $v$. We omit the subscript if the relevant complex hypergraph is clear from the context.

Let $\varepsilon>0$. We call a complex graph $G$ {\em $\varepsilon$-free} if $\iota(G)\leq \varepsilon \edeg_G v$ and $\edeg_G v>0$ for every vertex $v$.

A {\em subgraph} $G'$ of a complex hypergraph $G=(V,E_o,E_w,H,S)$ is a complex hypergraph $(V',E_o',E_w',H',S')$, such that $V'\subseteq V$, $E_o'\subseteq E_o$, $E_w'\subseteq E_w$, $H'\subseteq H$, $S'\subseteq S$ and stubs, edges and hyperedges of $G'$ are incident only with vertices in $V'$. We say that $G'$ is {\em spanning} if $V'=V$. We call two subgraphs {\em edge-disjoint} if their sets of edges, hyperedges and stubs are disjoint.

The {\em underlying hypergraph} of a complex hypergraph $G$ is the hypergraph obtained from $G$ by discarding the information about lengths and indices. Most notions for hypergraphs naturally extend to complex hypergraphs through underlying hypergraphs. Except for the two cases mentioned below, we use standard graph terminology for complex hypergraphs, looking at their corresponding underlying hypergraphs. For instance, we say that a complex hypergraph is {\em $k$-edge-connected} if its underlying hypergraph is $k$-edge-connected. Complex (hyper)graph terminology deviates from the usual terminology only on two occasions, the first one is the following definition of a complex path and the second one is the notion of shrinking that is introduced later.

A {\em complex path} is a sequence of vertices, wiggly and ordinary edges and stubs such that
\begin{enumerate}[label=(\roman*)]
\item the first and the last element of the sequence are either vertices or stubs,
\item a stub can appear only at the beginning or at the end of the sequence and is immediately followed or preceded by its incident vertex,
\item every edge is immediately followed and preceded by the vertices incident with it (distinct unless it is a loop),
\item vertices and edges alternate,
\item no stub or edge appears in the sequence more than once,
\item \label{cnd:6}there is a wiggly edge or more than $\ell$ ordinary edges between every two occurrences of the same vertex in the sequence, and 
\item two consecutive indices on the path are different unless they belong to the same wiggly edge or they are $-1$. That is, if we write indices of stubs and wiggly edges in the order given by the order in which stubs and edges (and their endpoints) appear in the complex path, two consecutive indices belonging to different wiggly edges or stubs must be different or $-1$. Note that there might be ordinary edges (which do not have indices) on the path between the wiggly edges or the stubs corresponding to consecutive indices.
\item We assume that the order of indices of a loop in a path is fixed. See Figure~\ref{fig:loop} for an example.
\end{enumerate}

\begin{figure} 
\begin{center}
\includegraphics[scale=1]{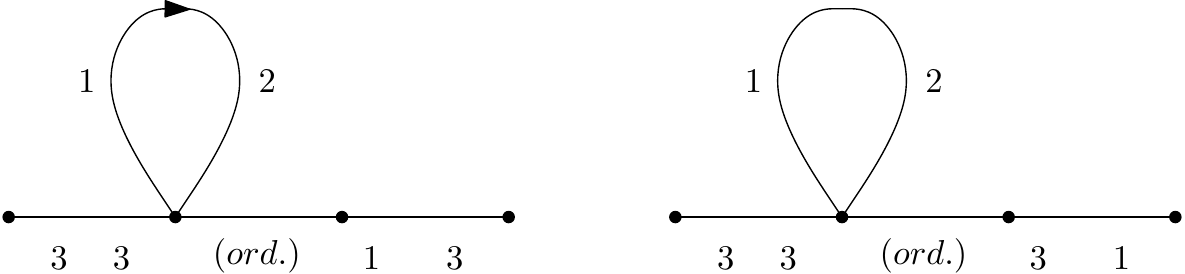}
\caption{Examples of complex paths with three wiggly edges and one ordinary edge. In the left example, only one ordering of the indices of the loop yields a complex path. The other ordering would lead to two consecutive indices $1$ not belonging to the same wiggly edge. In the right example, both orderings yield a complex path. We fix one of them as the "default" one.}
\label{fig:loop}
\end{center}
\end{figure}

We omit the word complex when there is no risk of confusion.

We call the first and the last element of the sequence the {\em ends} of the path. (These can be either stubs or vertices.)

Usually, we disregard the orientation of the path, i.e., the path $x_1,\ldots, x_n$ is considered to be the same as $x_n,\ldots, x_1$. Where  the orientation matters, we specify it by saying which element of the path is the first or the last one. If the first and the last element of the path $P$ is the same vertex (note that it cannot be a stub), we call the path a {\em loop-path} and we choose one orientation of the path to be the "default" one. After specifying an orientation of $P$, we refer to its other, reverse orientation as $P^r$.

For a fixed orientation of a path we define {\em the initial index and the terminal index} of the path as follows. If the path does not contain any stub or wiggly edge, we define both indices to be $-1$. Otherwise, we define the indices as the first and the last index in the sequence of indices appearing in the path, respectively. We sometimes refer to the initial and terminal index as the {\em index of the path at $v$}, where $v$ is the corresponding end of the path.
 
We define the {\em length} of a complex path $P$ as the sum of the lengths of the stubs and the edges appearing in $P$. We denote the length of a path $P$ by $\alpha(P)$. We call a complex path a {\em $k$-path} if its length is congruent to $k$ modulo $\ell$. We say that a complex path is {\em long}, if it contains at least one wiggly edge or at least $\ell$ ordinary edges.

For an edge $e$ with endpoints $u$, $v$ and two complex paths $P$ and $Q$ such that $u$ is the last element of $P$ and $v$ the first element of $Q$, $P \ct e\ct Q$ is formed by concatenating $P$, $e$ and $Q$. Similarly, for a stub $s$ incident with $v$ and a path $P$ such that $v$ is its first or its last element, we define $s\ct P$ or $P\ct s$, respectively, as the concatenation of $P$ and $s$. Abusing the notation, if $P$ and $Q$ are two paths with $v$ as the last and first vertex, respectively, we define $P\ct Q$ as the concatenation of $P$ and $Q$, without repetition of $v$. 
Note that the result of concatenation might not be a complex path, in particular, it might not satisfy the last three requirements in the definition of a complex path.

We say that a complex graph has an {\em $\ell$-path-decomposition}, if its edge set and stub set can be decomposed into disjoint $\ell$-paths.
Let $G=(V,E)$ be a simple graph treated as the complex graph $(V,E,\emptyset,\emptyset,\emptyset)$.
Observe that a complex $\ell$-path in $G$ might not be a path in the usual sense because of vertex repetitions. However, it can be decomposed into paths of length $\ell$ in the usual sense. It follows that if $G$ has an $\ell$-path-decomposition in the complex sense, $G$ has a decomposition into paths of length $\ell$ in the usual sense, because of condition~\ref{cnd:6}.


The {\em shrinking of a wiggly edge} $e$ with endpoints $v_1,v_2$ is the process of replacing $e$ by two stubs $s_1,s_2$, incident with $v_1$ and $v_2$, respectively, such that $\alpha(s_1)+\alpha(s_2)=\alpha(e)\mod \ell$ and $i^{s_1}_1=i^e_1$, $i^{s_2}_1=i^e_2$.
The {\em shrinking of a hyperedge} $h$ containing vertices $v_1,v_2$ and $v_3$ is the process of replacing $h$ either by a wiggly edge containing two of the vertices of $h$ and a stub incident with the third one such that the sum of their lengths is equal to  $\alpha(h)$, or by three stubs incident with the vertices of $h$ such that the sum of their lengths is equal to $\alpha(h)$, with indices of the wiggly edge or the stubs at $v_1,v_2,v_3$ equal to the corresponding indices of $h$.

A {\em shrinking} of a complex hypergraph $G$ is a complex hypergraph $G'$ obtained from $G$ by shrinking a  (possibly empty) set of wiggly edges and hyperedges. Note that shrinking does not change the degree of the vertices, $\alpha(G)=\alpha(G')$ and $\iota(G)=\iota(G')$. 
Moreover, shrinking is transitive, i.e., if $G'$ is a shrinking of $G$ and $G''$ is a shrinking of $G'$, $G''$ is a shrinking of $G$. 
A shrinking is {\em total} if it is a complex graph (i.e., all the hyperedges were shrinked). We say that a shrinking of a complex hypergraph is {\em absolute} if it does not contain any wiggly edges or hyperedges, i.e., it cannot be shrinked further.

We show that wiggly edges in a complex graph can be shrinked in a way which preserves $\ell$-path-decomposability of the complex graph.

\begin{observation}\label{obs:abs-shrink}
Let $G$ be a complex graph that has an $\ell$-path-decomposition and let $e$ be a wiggly edge of $G$. Then, $e$ can be shrinked in such a way that the resulting graph has an $\ell$-path-decomposition. Consequently, there exists an absolute shrinking of $G$ which has an $\ell$-path-decomposition.
\end{observation}

\begin{proof}
 Let $v_1$ and $v_2$ be the endpoints of $e$. Let $\PP$ be an $\ell$-path-decomposition of $G$ and let $P$ be the $\ell$-path in $\PP$ that contains $e$. Thus, $P=P_1 \ct e \ct P_2$ for some complex paths $P_1$ and $P_2$ with an end $v_1$ and $v_2$, respectively. Consider the complex graph $G'$ obtained by shrinking $e$ into stubs $s_1$ and $s_2$ incident with $v_1$ and $v_2$, respectively, with lengths satisfying $\alpha(s_1)=-\alpha(P_1) \mod \ell$ and $\alpha(s_2)=-\alpha(P_2) \mod \ell$. Since $P$ is an $\ell$-path, $\alpha(s_1)+\alpha(s_2)=-\alpha(P_1)-\alpha(P_2)=\alpha(e) \mod \ell$. Therefore, $G'$ is indeed a shrinking of $G$. Moreover, $P_1 \ct s_1$ and $P_2\ct s_2$ are $\ell$-paths. Thus, $\PP\cup \{P_1 \ct s_1,P_2\ct s_2\} \setminus \{P\}$ is an $\ell$-path decomposition of $G'$. 

Repeating this process for all wiggly edges in $G$ yields an absolute shrinking of $G$ which has an $\ell$-path-decomposition.
\end{proof}

Observe that two stubs $s_1$, $s_2$ incident with the same vertex with different indices or index $-1$ and such that  $\alpha(s_1)+\alpha(s_2)=0\mod \ell$ form a complex $\ell$-path. We call such a pair of stubs a {\em balanced stub-pair}. The following easy observation asserts that if a complex graph is $\ell$-path-decomposable after removing a balanced stub-pair, the original graph was $\ell$-path-decomposable as well. More generally, if $G$ has an $\ell$-path-decomposition and $\PP$ is a collection of $\ell$-paths, $G\dot{\cup} E(\PP)$ has an $\ell$-path-decomposition.

\begin{observation}\label{obs:balanced-stub}
Let $G$ be a complex graph and let $\{s_1,s_2\}$ be a balanced stub-pair incident with a vertex $v$. 
If the complex graph $G\setminus \{s_1,s_2\}$ has an $\ell$-path-decomposition $\PP$, $G$ has an $\ell$-path-decomposition consisting of $\PP$ and the $\ell$-path formed by $s_1$ and $s_2$.
\end{observation}


We say that a complex path $P$ starting with a vertex $v$ is {\em universal at $v$} if $P$ contains at least five edges, the first three edges of $P$ are loops incident with $v$ and changing the order of the first two loops yields a complex path $P'$ with the initial index different from the initial index of $P$. (Note that $P$ might be universal at $v$ and not $P^r$ or vice versa.) The terminal index of $P$ and $P'$ is the same. Note that a path without stubs can be universal at both its ends. Then, all four orderings of the first two and the last two loops in the path yield a complex path.

Let $G$ be a complex graph and $P$ be a complex path in $G$ consisting only of vertices and wiggly edges, and  containing at least one wiggly edge. The {\em condensing} of $P$ is the operation consisting of replacing $P$ by a wiggly edge $e_P$ of length $\alpha(P)$ with endpoints equal to the ends of $P$. We define the indices of $e_P$ to be $-1$ if $P$ is universal at the corresponding end, and to be equal to the respective index of $P$ otherwise.
We call a complex graph $G'$ obtained by a series of condensings in $G$ a {\em condensation} of $G$.

\begin{observation}\label{obs:loops}
Let $G$ be a complex graph and let $G'$ be a condensation of $G$. If there exists an $\ell$-path-decomposition of $G'$, there exists an $\ell$-path-decomposition of $G$.
\end{observation}

\begin{proof}
It is enough to prove the statement for $G'$ obtained from $G$ by condensing one path $P$ into the wiggly edge $e_P$. Consider an  $\ell$-path-decomposition of $G'$ and let $Q$ be the path in this decomposition containing $e_P$. Thus,  $Q=Q_1\circ e_P\circ Q_2$ for some (possibly empty) complex paths $Q_1, Q_2$ in $G'$. Then $Q'=Q_1\circ P\circ Q_2$ is an $\ell$-path in $G$.
\end{proof}

We can condense unit loops into loops with both indices $-1$ as follows.

\begin{observation}\label{obs:uni-path}
Let $G$ be a complex graph. Let $r=\iota(G)$ and let $v$ be a vertex of $G$ incident with at least $3r+\ell+4$ unit loops. For any given length $\ell'\leq \ell$, there exists a complex path $P$ of length $\ell'\mod \ell$ consisting of at most $\ell+4$ unit loops incident with $v$ such that both of its ends are universal or have index $-1$.
\end{observation}

\begin{proof}

Let $P'$ be a complex path of length $\ell'-4\mod \ell$ consisting of unit loops incident with $v$ and containing at least one loop. Observe that $P'$ can be constructed by greedily adding unit loops one by one, using at most $\ell$ of them, since $v$ is incident with at least $3r+\ell+4$ unit loops and therefore, after using less than $\ell$ loops, at least one of the remaining loops has one of its indices equal to $-1$ or different from the terminal index of the partially constructed $P'$. 

Let $i_1$ be the initial and $i_2$ be the terminal index of $P'$. Then, there exist four different unit loops $e,e',f,f'$ incident with $v$ that are not in $P'$ such that $e\ct e'\ct P'\ct f\ct f'$, $e\ct e'\ct P' \ct f'\ct f$, $e'\ct e\ct P'\ct f\ct f'$ and $e\ct e'\ct P' \ct f'\ct f$ are complex paths, each two of them differing in at least one index. It is sufficient to choose $e$ with indices $i^e_1,i^e_2$ different from $i_1$ or equal to $-1$, $e'$ with indices different from $i^e_1,i^e_2$ and $i_1$ or equal to $-1$ and similarly, $f$ with $i^f_1,i^f_2$ different from $i_2$ or equal to $-1$ and $f'$ with indices different from $i^f_1,i^f_2$ and $i_2$ or equal to $-1$. This is possible, because as long as $v$ is incident with more than $3r$ unit loops, it is possible to find a loop with indices different from three specified values or equal to $-1$.

\end{proof}

The purpose of the next lemma is to construct a condensation such that every vertex $v$ either has large unit degree, or all non-loops incident with it are ordinary or have index $-1$ at $v$.

\begin{lemma}\label{lem:condens}
Let $G$ be a $k$-edge-connected complex graph without stubs with $\udeg_G v \geq (1-\frac{1}{8(\ell-1)}) \deg_G v$ for every $v$ and minimum degree at least $7\iota(G)$. Then, there exists a complex graph $H$ obtained from a condensation of $G$ by deleting all the loops such that $H$ is $k$-edge-connected and for every vertex $v$, $\deg_H v=\deg_G v -\ldeg_G v$ and one of the followings holds 
\begin{enumerate}[label=(\alph*)]
\item\label{it:big-vert} $\udeg_{H} v \geq \udeg_{G} v/8\ell$ and $\udeg_{H} v \geq 3/4 \deg_{H} v$,
\item\label{it:univ-vert}  every edge of $H$ incident with $v$ is ordinary or unit with index $-1$ at $v$.
\end{enumerate} 
In particular, if $G$ has minimum degree $d$ and $\iota(G)<\varepsilon (1-\frac{1}{8(\ell-1)}) d/8\ell$, $H$ is $\varepsilon$-free.

\end{lemma}

\begin{proof}
Let $r:=\iota(G)$. We start with $H=G$ and modify it until $H$  satisfies the statement of the lemma. We treat the vertices of $H$ one by one. When treating a vertex $v$, we condense $1$-paths consisting of one non-loop and several loops incident with $v$ and then delete all the remaining loops incident with $v$. Before giving exact details of the process, we argue that we can indeed treat vertices one by one. In particular, we show that if a vertex $v$ satisfies~\ref{it:big-vert} or~\ref{it:univ-vert} after being treated, it satisfies~\ref{it:big-vert} or~\ref{it:univ-vert} at the end of the process as well.

Observe that after treating a vertex $v$, the degree of $v$ in the intermediate version of $H$ is equal to $\deg_G v -\ldeg_G v$  . Moreover, treating $v$ does not change the degree of other vertices and does not decrease the unit degree of other vertices. Also note that if $e$ is a non-loop with endpoints $u$ and $v$ and with index $-1$ at $u$, condensing a path consisting of $e$ and loops incident with $v$ yields a non-loop with index $-1$ at $u$. We do not modify ordinary edges in any way. It follows that we can treat vertices independently one after another.

Now, we describe treatment of the vertices in detail. We divide the vertices into two disjoint groups and treat vertices in each of the groups differently.

A vertex $v$ belongs to the first group if it satisfies $\deg_{G} v - \ldeg_{G} v \geq \udeg_{G} v/8\ell$ and to the second group otherwise. After being treated, the vertices in the first group will satisfy~\ref{it:big-vert} and the vertices in the second group will satisfy~\ref{it:univ-vert}. 

We treat every vertex $v$ in the first group as follows.
If more than three quarters of the non-loops incident with $v$ are unit, it is clearly enough to delete all the loops incident with $v$ to obtain a graph where $v$ satisfies the conditions of the lemma.
Otherwise, let $w$ be a non-unit non-loop incident with $v$.

We construct a path $P_w$ consisting of at most $\ell-1$ unit loops incident with $v$ of length $1-\alpha(w)\mod \ell$ such that $w\ct P_w$ is a complex path. Condensing $w\ct P_w$ yields a unit non-loop. We repeat this process for non-unit non-loops incident with $v$ as long as possible.
Note that the number of non-loops incident with $v$ does not change by this procedure and the graph stays $k$-edge-connected. The process terminates either when there is no non-unit non-loop incident with $v$, or it is impossible to construct $P_w$. We claim that the latter never happens.

Observe that as long as there are at least $r/2+\ell-1$ unit loops incident with $v$, a complex path $P_w$ consisting of at most $\ell-1$ unit loops of length  $1-\alpha(w) \mod \ell$ with the initial index $-1$ or different from the index of $w$ at $v$, can be constructed greedily for any $w$ incident with $v$. Let $n=\deg_G v -\udeg_G v$, $u_0$ the number of unit non-loops incident with $v$ and $u_1$ the number of unit loops incident with $v$. 
It is enough to show that $n(\ell-1)\leq u_1- r/2$. Then, $P_w$ can be constructed for every non-unit non-loop $w$.

\begin{claim}
It holds that $n(\ell-1)\leq u_1- r/2$.
\end{claim}

\begin{inproof}
First, we show that $u_1\geq \deg_G v/4$. 

Observe that $n+u_0+2u_1=\deg_G v$.
By our assumptions, $u_0\leq 3n$ and $n\leq (1-(1-\frac{1}{8(\ell-1)}))\deg_G v$. Consequently, $4\frac{1}{8(\ell-1)}\deg_G v+2u_1\geq \deg_G v$. Thus, $u_1\geq \frac{1-\frac{1}{2(\ell-1)}}{2}\deg_G v =\frac{2(\ell-1)-1}{4(\ell-1)} \deg_G v $. Recall that we assume that $\ell\geq 2$. Therefore, we have $u_1\geq \deg_G v /4$.

We also have $n\leq \frac{\deg_G v}{8(\ell-1)}$.
Thus, $n(\ell-1)\leq \deg_G v/8\leq \deg_G v/4 - r/2\leq u_1- r/2$, because $\deg_G v\geq 7r$.

\end{inproof}

We treat every vertex $v$ in the second group as follows.
We show that $v$ is incident with at least $3r+(\ell+4)\udeg_{G} v/(8\ell)$ unit loops.  Then, by repeated application of Observation~\ref{obs:uni-path}, we construct a path $P_w$ such that condensing $w\ct P_w$ yields a unit edge with index $-1$ at $v$ for every wiggly non-loop $w$ incident with $v$.

Since there are at most $\udeg_{G} v/(8 \ell)$ non-loops incident with $v$ (as $v$ is in the second group), the number of unit loops incident with $v$ is at least  $\frac{8\ell-1}{8 \ell}\udeg_{G} v$.  This is greater than $\udeg_{G} v/(8 \ell)\cdot(\ell+4)+3r$, because, using that $1-\frac{1}{8(\ell-1)}\geq 7/8, \ell\geq 2$ and $deg_G v\geq 7r$, we have 
\begin{align*}
\frac{8\ell-1}{8 \ell}\udeg_{G} v-\left(\frac{\ell+4}{8 \ell}\udeg_{G} v +{3r}\right)
&=\frac{7\ell-5}{8 \ell}\udeg_{G} v-{3r}\\
&\geq \left(1-\frac{1}{8(\ell-1)}\right) \deg_G v \left(\frac{7}{8}-\frac{5}{8\ell}\right)-{3r}\\
&\geq 7/8\cdot 7r\cdot 9/16 - 3r\\
&=(441/128-3)r>0
\end{align*}

\end{proof}

We conclude this section by showing that all the loops in a $1/4$-free graph with one vertex can be concatenated into a single path. 

\begin{lemma}\label{lem:singleton}
Let $G$ be a complex graph without stubs that has only one vertex, $L$ loops and $\iota(G)\leq L/2$. Then all the loops of $G$ can be concatenated to form a single complex path. 
\end{lemma}

\begin{proof}
Let $R$ be the set of loops incident with $v$ such that both of their indices are the same and different from $-1$ (that is, $i_1^e=i_2^e\neq -1$ for every $e\in R$) and let $R'$ be the set of the remaining loops. 
If more than $|R|/2$ loops in $R$ have the same index $i$, let $S$ be the set of these loops. Note that $|S|\leq L/4$ and there are at most $\iota(G)-2|S|\leq L/2-2|S|$ loops with index $i$ in $R'$. Thus, there is a set $T$ of loops in $R'$ that do not have index $i$ such that $|T|=|S|-|R\setminus S|-1$. Otherwise, let $T=\emptyset$.

Let $e_1,\ldots, e_n$ be a sequence of loops in $R\cup T$, with the loops in $R$ first, ordered from the most frequent index in $R$, to the least frequent, followed by loops of $T$ in any order. 

Observe that the loops of $R\cup T$ in the order $e_1,e_{\lceil n/2\rceil +1}, e_2, e_{\lceil n/2\rceil +2}, \ldots,$  $e_i, e_{\lceil n/2\rceil +i},\ldots$ form a complex path.

Now, observe that all the remaining loops incident with $v$ can be appended to this path greedily, since each of them has two distinct indices (thus, one of the indices is different from the terminal index of the path to which we want to append the loop) or indices $-1$.

\end{proof}

\section{Tools}\label{sec:tools}
In this section, we introduce several tools used in the rest of the paper. First, we focus on connectivity in hypergraphs and state generalizations of some classical results for graphs.

The following Lemma~\ref{cor:sparse-tree} was essentially proven in~\cite{bib:thomassen-paths}. Our version (which we prove in~\cite{bib:our-trees}) differs in some details. 


\begin{lemma}\label{cor:sparse-tree}
For every $\varepsilon>0$ and integer $m$, there exists $L$ such that every $2^n$-edge-connected stub graph $G$ with minimum degree at least $L$, where $n=2+m+ \lceil\log (1/\varepsilon)\rceil$,   
has $2^{m}$ edge-disjoint spanning trees $T_1,\ldots, T_{2^m}$ such that
\[\sum_{i\in [2^m]}\deg_{T_i} v \leq \varepsilon \deg_{G} v\]
 for every $v\in V(G)$.
\end{lemma}

A {\em bridge} in a graph or a hypergraph is an edge such that after removing it, the graph or the hypergraph becomes disconnected.

A hypergraph $H$ is {\em $k$-partition-connected} if for any partition of $V(H)$ into $p$ non-empty parts, the number of hyperedges 
of $H$ containing vertices in more than one part is at least $k(p-1)$. The following easy observation links partition-connectivity to edge-connectivity.

\begin{observation}\label{obs:part-con}
A $3k$-edge-connected hypergraph $H$ with edges of size at most three is $k$-partition-connected. 
\end{observation}

\begin{proof}
Let $\HH$ be a partition of $H$ into $p$ nonempty parts. Since $H$ is $3k$-edge-connected, every part of $\HH$ is connected to other parts by at least $3k$ edges. Each such edge contains vertices in at most three parts of $\HH$, therefore in total there are at least $3kp/3=kp$ edges containing vertices in more than one part of $\HH$.
\end{proof}

A {\em spanning hypertree} of a hypergraph $(V,E)$ is a hypergraph $(V,F)$, where $F\subseteq E$, $|F|=|V|-1$, $\bigcup_{e\in F}e=V$ and there are at most $|X|-1$ hyperedges $e\in F$ such that $e\subseteq X$ for every $X\subseteq V$.
The classical result of Nash-Williams and Tutte about the number of edge-disjoint spanning trees 
 was generalized to hypergraphs in~\cite{bib:nwh} as follows.

\begin{theorem}[Nash-Williams for hypergraphs~\cite{bib:nwh}]\label{thm:nwh}
A hypergraph $H$ is $k$-partition-connected if and only if $H$ contains $k$-hyperedge-disjoint spanning hypertrees.

\end{theorem}

Thus, every $1$-partition-connected hypergraph contains a spanning hypertree. By a result of Lov\'asz~\cite{bib:lovasz}, it is possible to shrink a spanning hypertree in such a way that the resulting graph is a spanning tree.


Let $H$ be a hypergraph. We call an ordered partition $(A,B)$ of the vertex set of $H$ into two parts a {\em cut} in $H$. We denote by $E(A,B)$ the set of hyperedges of $H$ incident with vertices in both $A$ and $B$ and call $|E(A,B)|$ the {\em order of the cut} $(A,B)$.
We define the {\em value} of the cut $(A,B)$ as $val(A,B)=\sum_{e\in E(A,B)}|e\cap A|$. Note that this definition is asymmetric, i.e., $val(A,B)$ and $val(B,A)$ are not equal in general. Next, we prove an analogue of Lemma~\ref{lem:cut} for hypergraphs.

\begin{lemma}\label{lem:hycut}
Let $k$ be an integer and $H$ be a hypergraph with edges of size at most three. Then there is a cut $(A,B)$ in $H$ of value at most $3k$ (and thus, of order at most $3k$) such that $H[A]$ is $k$-edge-connected or has only one vertex.
\end{lemma}

\begin{proof}
If $H$ is $k$-edge-connected, we are done. Otherwise, we consider a cut $(A,B)$ with value at most $3k$ such that $A$ is inclusion-wise minimal. (Such a cut exists in $H$, because $val(A,B)\leq 2|E(A,B)|$ for every cut $(A,B)$.) Consider a cut $(A_1,A_2)$ in $H[A]$ and let $r$ be its order. We show that $r>k$. We have
\[ val(A_1,V\setminus A_1) + val(A_2,V\setminus A_2)=val(A,B)+\sum_{e\in E(A_1,A_2)}|e|.\]
Note that $val(A,B)\leq 3k$ and $\sum_{e\in E(A_1,A_2)}|e|\leq 3r$. Moreover, $val(A_1,V\setminus A_1), val(A_2,V\setminus A_2)>3k$ by minimality of $A$.  It follows that $\sum_{e\in E(A_1,A_2)}|e|>3k$ and thus, $r>k$. Since this is true for any cut in $H[A]$, $H[A]$ is $k$-edge-connected.
\end{proof}




Let $G$ be a hypergraph (or a multigraph). We denote the minimum number of hyperedges in a cut separating vertices $u$ and $v$ in $G$ by $\lambda_{G}(u,v)$.

{\em Splitting off} a pair of edges $su$, $sv$ in a multigraph $G$ is replacing these two edges by a new edge $uv$. We define splitting off a pair of hyperedges of size two in a hypergraph in the same way.

\begin{theorem}[Mader splitting off theorem~\cite{bib:mader}]\label{thm:mader}

Let $G=(V,E)$ be a connected multigraph and let $s$ be a vertex such that the degree of $s$ is not three and $s$ is not incident with bridges.
Then there are two edges $e,f$ incident with $s$ such that the multigraph $G_{e,f}$ obtained by splitting off $e$ and $f$ satisfies $\lambda_{G_{e,f}}(u,v)=\lambda_{G}(u,v)$ for any $u,v\in V\setminus \{s\}$.
\end{theorem}

\begin{corollary}\label{cor:mader}
Let $H$ be a connected hypergraph with at least two vertices and hyperedges of size at most three and let $s$ be a vertex that is not contained in any hyperedge of size three and not incident with bridges. Then there are two hyperedges $e,f$ of size two incident with $s$ such that the hypergraph $H_{e,f}$ obtained from $H$ by splitting off $e$ and $f$ satisfies $\lambda_{H_{e,f}}(u,v)=\lambda_{H}(u,v)$ for any $u,v\in V\setminus \{s\}$.
\end{corollary}

\begin{proof}
Let $G$ be a hypergraph with hyperedges of size at most three.
Let $\tau$ be the operation creating a multigraph $\tau(G)$ from $G$ by replacing each hyperedge $h$ of size $3$ containing vertices $u_1,u_2$ and $u_3$ by a vertex $h$ adjacent to $u_1$, $u_2$ and $u_3$ (with possible multiplicities if $u_i=u_j$ for some $i\not= j$).
Observe that $\lambda_{G}(u,v)=\lambda_{\tau(G)}(u,v)$ for every $u,v\in V(G)$. In particular, if $s$ is not incident with a bridge in $G$, it is not incident with a bridge in $\tau(G)$. 

Therefore, we can apply Theorem~\ref{thm:mader} to $\tau(H)$ and the vertex $s$. Let $e,f$ be a pair of edges incident with $s$ in $\tau(H)$ as in Theorem~\ref{thm:mader} and let $\tau(H)_{e,f}$ be the multigraph obtained from $\tau(H)$ by splitting off $e$ and $f$. Then, $\lambda_{\tau(H)_{e,f}} (u,v)=\lambda_{\tau(H)}(u,v)$ for every $u,v\in V(\tau(H))\setminus \{s\}$.

Note that since $s$ is not contained in any hyperedge of size three in $H$, $e$ and $f$ are hyperedges of size two in $H$. Moreover, $\tau(H)_{e,f}$ is isomorphic to $\tau(H_{e,f})$. It follows that $\lambda_{H_{e,f}}(u,v)=\lambda_{H}(u,v)$ for any $u,v\in V(H)\setminus \{s\}$.  
\end{proof}

Let $G$ be a complex graph without stubs. We say that a spanning subgraph $H$ of $G$ is an {\em $\varepsilon$-fraction of $G$} if it satisfies $|\deg_H v- \varepsilon\deg_G v|\le 2$ for every vertex $v$ of $G$. 
The following lemma shows that there exists an $\varepsilon$-fraction of $G$ for any  $\varepsilon=1/2^c$, $c\in \NN$. 

\begin{lemma} \label{lem:fraction}
Let $m_1, \ldots, m_k$ be such that $\sum_{i\in [k]} 2^{-m_i}\leq 1$ and let $G$ be a complex graph without stubs. Then, there exist edge-disjoint subgraphs $G_1,\ldots,G_k$ of $G$ such that $G_i$ is a $2^{-m_i}$-fraction of $G$ for every $i\in [k]$.
\end{lemma}

\begin{proof}

Let $m=\max_{i\in [k]} m_i$ 
and let $H$ be the underlying multigraph of $G$. 

Let $\bar{H}$ be the multigraph obtained from $H$ by adding an auxiliary vertex $s$ adjacent to all vertices of odd degree in $H$. Since the number of such vertices is even, all vertices of $\bar{H}$ have even degree. 
Consider an Eulerian tour in each connected component of $\bar{H}$ and color its edges by two colors, $0$ and $1$, in an alternating way, starting from $s$ if the connected component contains $s$.

Let $H_0$ be the subgraph of $H$ of color $0$ and $H_1$ the subgraph of $H$ of color $1$. We apply the procedure described above for $H_0$ and $H_1$ to obtain graphs $H_{00}$, $H_{01}$, $H_{10}$ and $H_{11}$ and repeat it recursively until we obtain a collection $\{H_Q\}_{Q\in \{0,1\}^m}$ of edge-disjoint spanning subgraphs of $H$. 

 Let $v$ be a vertex of $H$ and let $d$ be its degree. Observe that $\deg_{H_0} v,\deg_{H_1} v \in [d/2-1,d/2+1]$ and this is true for every iteration of the procedure. Thus, 
  \[\deg_{H_Q} v\leq \underbrace{(\ldots (}_{i \times} d/\underbrace{2 +1)/2 +1)/2 \ldots )/2+1}_{i \times}\leq d/2^i+\sum_{j=0}^{i-1} 1/2^{j}\leq d/2^i +2\]
for every $H_{Q}$ with $Q\in \{0,1\}^{i}$.
 Analogously, $\deg_{H_Q} v \geq d/2^i-2$ for every $H_{Q}$ with $Q\in \{0,1\}^{i}$.

We choose $Q_i\in \{0,1\}^{m_i}$ for every $i\in [k]$ so that the graphs $H_{Q_i}$ are pairwise edge-disjoint. We let $G_i$ be the subgraph of $G$ with underlying multigraph $H_{Q_i}$ for every $i\in [k]$.  
Consequently, all $G_i$'s are edge-disjoint $2^{-m_i}$-fractions of $G$.  

\end{proof}

We conclude the section by stating the Asymmetric Lov\'asz Local Lemma, which we extensively use in the following sections.

\begin{lemma}[Lov\'asz Local Lemma (asymmetric version)~\cite{bib:lll}]\label{LLL}
Let ${\mathcal {A}}=\{A_{1},\ldots, A_{n}\}$ be a finite set of events in a probability space $\Omega$. For each $A\in {\mathcal {A}}$, we let $\Gamma (A)$ denote a subset of ${\mathcal {A}}$ such that $A$ is independent from the collection of events ${\mathcal {A}}\setminus (\{A\}\cup \Gamma (A))$. If there exists an assignment of reals $x:{\mathcal {A}}\to (0,1)$ to the events such that  for every $A\in {\mathcal {A}}$:
\[\Pr(A)\leq x(A)\prod_{B\in \Gamma (A)}(1-x(B)),\]
then the probability of avoiding all events in ${\mathcal {A}}$ is positive, in particular

    \[\Pr \left(\bigwedge_{A\in {\mathcal {A}}}{\overline {A}}\right)\geq \prod \nolimits _{A\in {\mathcal {A}}}(1-x(A)).\] 
\end{lemma}

\section{Shrinking the hypergraph}\label{sec:shrink}
In this section, we study shrinking of hypergraphs and complex graphs. In~\cite{bib:lovasz}, Lov\'asz showed that one can shrink any hypertree to a tree. We show that if a hypertree has hyperedges of size at most three, it is possible to shrink it to a tree in such a way that the degree of each vertex in the resulting tree is at least a fraction of $1/100$ of its degree in the original hypertree. Then, we use this result to shrink a complex hypergraph in such a way that the resulting complex graph is highly connected and of high edge-degree.

\begin{lemma}\label{lem:simple-shrink}
Let $H$ be a hypertree with hyperedges of size at most three. It is possible to shrink all hyperedges of size three so that the resulting graph is a spanning tree $T'$ such that $\deg_{T'} v\geq \deg_H v/100$ for every $v\in V(H)$.
\end{lemma}

We can derive the following corollary for complex hypergraphs. It is enough to apply  Lemma~\ref{lem:simple-shrink} to the underlying hypergraph and perform analogous shrinking on the complex hypergraph.

\begin{corollary}\label{cor:simple-shrink}
Let $G$ be a complex hypergraph such that its underlying hypergraph is a hypertree. It is possible to shrink all hyperedges so that the resulting complex graph $G'$ is connected (in particular, its edges form a spanning tree) and satisfies $\edeg_{G'} v\geq \deg_G v/100$ for every $v\in V(G)$.
\end{corollary}

In the proof of Lemma~\ref{lem:simple-shrink}, we use the following technical lemma for bounding the number of digits of numbers with a given sum. In the lemma and throughout the paper we use $\log$ to denote  logarithm to base two.
\begin{lemma}\label{lem:technical}
Let $d_1, \ldots, d_m$, $d$ and $n$ be positive integers such that $\sum_{i\in [m]} d_i\leq d$ and $n\geq m$. Then $\sum_{i\in [m]} \log d_i \leq m\log(d/m)$ and $\sum_{i\in [m]}\log d_i\leq n \log (d/n+1)$.
\end{lemma}

\begin{proof}
Recall that by the inequality between arithmetic and geometric mean 
\[\left(\prod_{i\in[m]} d_i\right)^{1/m} \leq \frac{\sum_{i\in [m]}d_i}{m}.\]

Since $\prod_{i\in[m]} d_i = 2^{\sum_{i\in[m]} \log d_i}$, we get
\[ \sum_{i\in[m]} \log d_i \leq m \log \left(\frac{\sum_{i\in [m]}d_i}{m} \right) \leq m \log \left(d/m\right).\]

Define $d_i:=1$ for $m<i\leq n$. Note that $\sum_{i\in[m]} \log d_i= \sum_{i\in[n]} \log d_i$ and $\sum_{i\in [n]} d_i\leq d+n$. As in the previous case, we get
\[ \sum_{i\in[n]} \log d_i \leq n \log \left(\frac{\sum_{i\in [n]}d_i}{n} \right) \leq n \log \left((d+n)/n\right)=n \log(d/n+1).\]
\end{proof}

\begin{proof}[Proof of Lemma~\ref{lem:simple-shrink}.]
Note that since $H$ is a hypertree, it does not have edges of size one. By the previously mentioned result of Lov\'asz~\cite{bib:lovasz}, it is always possible to shrink hyperedges of size $3$ in $H$ so that the resulting graph is a spanning tree. Let $T$ be such a spanning tree, rooted in an arbitrarily chosen vertex $r$. In the following claim, we describe a way to modify $T$, while preserving the property that it can be obtained from $H$ by shrinking hyperedges of size $3$.

\begin{claim}
Let $h=uvw$ be a hyperedge in $H$ such that $T$ contains $uv$ obtained by shrinking $h$. Then, we can replace $uv$ by $uw$ or $vw$ so that the resulting graph $T'$ is a spanning tree which can be obtained from $H$ by shrinking hyperedges.
\end{claim}

\begin{inproof}
Without loss of generality assume that $u$ is on the path from $v$ to $r$ in $T$. If $w$ is in the subtree of $T\setminus uv$ which contains $v$, we replace $uv$ by $uw$. Otherwise, we replace $uv$ by $vw$. In both cases we added an edge between the two components of $T\setminus uv$ and therefore obtained a spanning tree. Since the edge that replaces $uv$ can be obtained by shrinking $h$, the resulting spanning tree can be obtained from $H$ by shrinking hyperedges.
\end{inproof}

Note that exactly one of the two possible edge replacements in the previous claim yields a tree. We call this operation {\em flipping $h$ to $w$}. In the following claim, we observe that the outcome of flipping several hyperedges to a fixed vertex $w$ does not depend on the order in which we perform these flippings.

\begin{claim}
Let $h=[u,v,w]$ be a hyperedge in $H$ such that $T$ contains $e=uv$ obtained by shrinking $h$. Let $T'$ be a tree obtained  from $T$ by flipping some hyperedges different from $h$ to $w$. The outcome of flipping $h$ in $T$ and $T'$ is the same, i.e., $e$ is replaced by $uw$ in $T$ if and only if $e$ is replaced by $uw$ in $T'$.
\end{claim}

\begin{inproof}
Orient the edges of $T$ to obtain an outbranching with $w$ being the only vertex of indegree $0$. Observe that the result of flipping $h$ to $w$ depends only on the orientation of the edge $e$ (in particular, the edge $e$ will be replaced by the edge between $w$ and the head of the arc $\vec{e}$). Moreover, observe that flipping hyperedges different from $h$ to $w$ does not change the orientation of the edge $e$.
\end{inproof}

We call a vertex $v$ {\em poor} if $\deg_{T}v<\deg_{H}v/100$. Note that since every vertex has degree at least one in $T$, every poor vertex has degree at least $100$ in $H$.
Starting with $T$ we run the following randomized algorithm. If $T$ contains a poor vertex $v$ we run the following procedure $\Fix(v)$ modifying $T$. For every hyperedge $h=uvw$ of $H$ such that $uw\in E(T)$ we take a random bit and if it is $1$, we flip $h$ to $v$. We then recursively run $\Fix(v')$ for every poor vertex $v'$ in the closed neighborhood of $v$ in $H$ (that is, the neighborhood including $v$). Note that after the end of the procedure $\Fix(v)$, $v$ is not poor and the vertices which were not poor before did not become poor. Therefore, we need to call $\Fix(v)$ for every vertex at most once to obtain a tree without poor vertices. However, we do not know how many times can $\Fix(v)$ be called in the recursion.

Now, we argue that $\Fix(v)$ terminates with probability one using the entropy compression argument, that was first introduced in~\cite{bib:entropy}. Each call of $\Fix(v)$ uses at least $99\deg_{H}v/100$ random bits, because $\Fix$ is called only if $v$ is poor. We record the run of the algorithm in such a way that at any moment, it is possible to reconstruct the original tree $T$ and the sequence of random bits used from the current state of $T$ and the algorithm's record.

We assume that the information of the initial vertex is not part of the algorithm's record. Moreover, we assume that we are given an ordering on the vertices of $H$. For recording, we use four symbols; $0,1$, $\Downarrow$ for a recursive call and $\Uparrow$ for return from recursion. Therefore, we need two bits per symbol. For each call $\Fix(v)$, we record

\begin{enumerate}[label=(\alph*)]
\item\label{it:orig} what was the neighborhood of $v$ in $T$ before flipping
\item\label{it:v} vertices in $N_H[v]$ for which we run $\Fix$ recursively, each separated by $\Downarrow$ from both sides
\item\label{it:re} return from the recursive call $\Uparrow$
\end{enumerate}

The record after calling $\Fix(v)$, with $N_H[v]=\{v_0,\ldots, v_{\deg_H v}\}$ might for example look as follows:
\[(\mbox{\tt{state of }} v) \Downarrow v_{i_1} \Downarrow (\mbox{\tt{state of }} v_{i_1}) \Uparrow \Downarrow v_{i_2} \Downarrow (\mbox{\tt{state of }} v_{i_2})\ldots\]

It is easy to see that from \ref{it:v} and \ref{it:re} (and knowledge of the initial vertex $v$) it is possible to reconstruct for which vertices and in which order $\Fix$ was called. Next, from the current $T$, we can reconstruct the initial $T$ by modifying neighborhoods of vertices according to the records from \ref{it:orig}. 

We encode \ref{it:orig} using symbols $0$ and $1$. We claim that we can always do so using $\lceil \deg_H v/7\rceil$ symbols. 
 
\begin{claim}
The neighborhood of a poor vertex $v$ in $T$ can be encoded into a binary string of length $\lceil \deg_H v/7\rceil$.
\end{claim}

\begin{inproof}
First, we create a binary string of length $\deg_H v$ by writing $0$ for a vertex not adjacent to $v$ in $T$ and $1$ for the vertex adjacent to $v$ (following the given ordering on the vertices). Note that this string contains at most $1/100 \deg_H v$ ones. We can encode this string using $3$ symbols, $1$ and $\mathbf{0,1}$; we replace each interval consisting of $0$ by its length encoded using symbols $\mathbf{0,1}$. 

Encoding $d$ needs at most $1+\lfloor\log d\rfloor$ symbols. Thus, by Lemma~\ref{lem:technical}, the number of symbols $\mathbf{0,1}$ that are needed is bounded by $ 1/100 \deg_H v (1+ \log 101)$. It follows that the length of the resulting string is at most $ 1/100 \deg_H v \cdot(2+ \log 101)$. There are at most $3^{ 1/100 \deg_H v (2+ \log 101)}$ such strings. They can be translated into binary strings of length \[ \lceil \log 3^{ 1/100 \deg_H v (2+ \log 101)}\rceil= \lceil  1/100 \deg_H v (2+ \log 101) \log 3\rceil \leq \lceil \deg_H v/7\rceil.\]

\end{inproof}

When recording \ref{it:v}, we consider ordering on the vertices in the closed neighborhood of $v$ in $H$ and we encode a $j$-th recursive call from $v$ to its neighbor $v_{i_j}$ as the difference between the order of $v_{i_j}$ and the previously called neighbor $v_{i_{j-1}}$ of $v$. That is, $\Fix$ is called on the poor neighbors of $v$ according to the given ordering (this is possible, because, as previously observed, when $\Fix$ returns, the vertices that were previously not poor did not become poor). For the first recursive call from $v$, we record the order $i_1$ of the first poor neighbor to recurse to, for the consecutive calls, instead of recording $i_j$ for the $j$-th vertex $\Fix$ recurses to from $v$, we record $i_j-i_{j-1}$.

Later in the entropy compression argument, we argue that every recursive call uses more random bits than the size of the algorithm's record (in bits). To make this argument work, we split the "cost" of recording recursive calls from $v$ between the record of the call $\Fix(v)$ and records of the recursive calls originating from the call $\Fix(v)$. In particular, we request that every recursive call contributes $\ec=5$ symbols for being called. 

We claim that with the contributions from the recursive calls, there are at most $\deg_H v/7$ symbols additional symbols needed to record the list of all the recursive calls initiated by $\Fix(v)$.

\begin{claim}
Assume that $\Fix(v)$ initiates $k$ recursive calls. Then, we need at most $\deg_H v/7 + k\ec$ symbols for recording the list of these calls.
\end{claim}

\begin{inproof}
The recording of the recursive calls from $v$ consists of $2k$ symbols $\Downarrow$ and $k$ binary numbers $d_1,\ldots, d_k$, such that $\sum_{i\in [k]} d_i\leq \deg_H v$. We need $\lfloor\log d_i\rfloor +1$ symbols   for recording $d_i$ in binary. Thus, we need $3k+ \sum_{i\in [k]} \lfloor\log d_i\rfloor$ symbols in total. Using Lemma~\ref{lem:technical}, we get that $3k + \sum_{i\in [k]} \log d_i\leq 3k + k (\log \deg_H v - \log k)= k \ec + k (\log \deg_H v - \log k +3-c)$. Since $k (\log \deg_H v - \log k +3-c)$ attains its maximum when $\log k =\log \deg_H v-2-1/\ln 2$, $k (\log \deg_H v - \log k +3-\ec)< \deg_H v/7$ for any $k$. Thus, $3k+ \sum_{i\in [k]} \lfloor\log d_i\rfloor\leq \deg_H v/7 + k\ec$ for every $k$.
\end{inproof}

Therefore, the number of symbols of the record corresponding to the call $\Fix(v)$ is less than $d=\lceil 1/7 \deg_H v\rceil+1/7 \deg_H v + c + 1 \leq 2/7 \deg_H v + 7$ ($\lceil 1/7 \deg_H v\rceil$ symbols for recording the neighborhood of $v$, $1/7 \deg_H v$ symbols for recording the recursive calls, $c$ for $\Fix(v)$ being called and one symbol for $\Uparrow$ for return from $\Fix(v)$). That is at most $2d \leq 4/7 \deg_H v + 14$ bits. Since $\deg_H v$ is at least $100$ for every poor vertex $v$, $4/7 \deg_H v + 14<\frac{99}{100}\deg_H v$.
It follows that the algorithm terminates almost surely.
\end{proof}

\begin{lemma}\label{lem:shrink}
For every positive integers $k$, $r$ and $D\geq 10^6$ the following holds. Let $H$ be an $\ell$-divisible $3k$-edge-connected complex hypergraph without stubs and with
\begin{itemize}
\item minimum degree $D$,
\item $\iota(G)\leq r$, and
\item  $\udeg v\geq  \edeg v/500 $ or $\hdeg v \geq \deg v/2$ for every $v$.
\end{itemize}

Then, there is a total shrinking $H'$ of $H$ that is $k$-edge-connected, with minimum edge degree at least $D/200$, $\udeg_{H'} v\geq \deg_{H'} v/1000 $ and all except $S:=\max(r,1000)$ stubs at each vertex can be divided into balanced stub-pairs. Note that $H'$ is also $\ell$-divisible.
\end{lemma}

\begin{proof}[Proof of Lemma~\ref{lem:shrink}]
First, note that if $\hdeg v < \deg v/2$, we have $\edeg_{H'}v\geq D/200$ and $\udeg_{H'} v\geq \deg_{H'} v/1000$ for any total shrinking $H'$ obtained from $H$ by shrinking only hyperedges (but not wiggly edges).

By Observation~\ref{obs:part-con} and Theorem~\ref{thm:nwh}, applied to the underlying hypergraph of $H$, $H$ contains $k$ edge-disjoint spanning hypertrees $T_1,\ldots, T_{k}$. By Corollary~\ref{cor:simple-shrink}, the hyperedges of each hypertree $T_i$, $i\in [k]$, can be shrinked so that the resulting complex graph $T'_i$ consists of a spanning tree and possibly some stubs and satisfies $\edeg_{T_i'} v\geq \deg_{T_i} v/100$ for every $v\in V(H)$.

Let $R=H\setminus (\bigcup_{i=1}^{k} E(T_i))$. We shrink the hyperedges of $R$ to obtain a complex graph $R'$ such that $\edeg_{R'} v\geq \deg_R v/2$ for every vertex with $\deg_R v\geq D/2$.

The existence of such a shrinking follows from Chernoff bound and Lov\'asz Local Lemma; if we shrink hyperedges in $R$ independently at random, the expected edge-degree of $v$ in $R'$ is at least $2/3 \deg_R v$ for every vertex $v$. Let $E_v$ be the event that $\edeg_{R'} v$ is less than $\deg_R v/2$. By Chernoff bound, the probability of $E_v$ is less than $c^{\deg_R v}$ with $c< 0.98$.  
Each $E_v$ is independent of all but at most $2\deg_R v$ other events, therefore, by asymmetric Lov\'asz Local lemma with $x=1/\deg_R v$, the probability that $E_v$ does not happen for any $v$ with $\deg_R v\geq D/2$ is positive, since $c^{\deg_R^v}<1/\deg_R v (1- 2/D)^{\deg_R v}$ for $v$ with $\deg_R v > D/2$ and $D$ sufficiently large.

Therefore, the complex graph $H'=R'\cup \left(\bigcup_{i\in[\kk]} (T'_i)\right)$ has minimum edge-degree at least $D/200$. We call a pair of a wiggly edge and a stub obtained by shrinking a hyperedge {a \em shrink-pair}. Next, we assign lengths to stubs and wiggly edges in shrink-pairs so that the conditions on unit degree and stubs are satisfied. Again, this follows from Lov\'asz Local lemma.

We divide the shrink-pairs in $H'$ into {\em rigid} and {\em flexible} (calling the edge and the stub in a rigid or a flexible shrink-pair rigid or flexible, respectively) in such a way that each vertex $v$ incident with at least $\deg_H v/2$ hyperedges in $H$ is incident with at least $\deg_H v/1000$
 rigid wiggly edges in $H'$ and for each vertex that has stub-degree greater than $S$, at least half of the stubs are flexible.

From such an assignment, the result already follows; we assign length one to the rigid wiggly edges and {\em complementary} length to the corresponding rigid stubs; that is, the length such that the sum of lengths of the wiggly edge and the stub in the shrink-pair is equal to the length of the original hyperedge. Then, for every vertex $v$ of stub degree greater than $S$, we proceed as follows. We greedily pair each rigid stub incident with $v$ with a flexible stub incident with $v$ of different index or index $-1$. After all rigid stubs are paired, we greedily pair remaining flexible stubs with different indices or indices equal to $-1$. Since at least half of the stubs incident with $v$ are flexible, there are at most $r$ stubs with the same index, and since $S\geq r$, at most $S$ stubs incident with $v$ remain unpaired.

We assign lengths to flexible stubs in the pairs in such a way that these pairs are balanced and we assign arbitrary lengths to unpaired flexible stubs. Finally, we assign complementary lengths to the corresponding flexible edges. By construction, the resulting complex graph $H'$ satisfies the conditions of the lemma.

We conclude the proof by showing that it is possible to divide the shrink-pairs in $H'$ into {\em rigid} and {\em flexible} in the desired way. 

\begin{claim}
It is possible to divide the shrink-pairs in $H'$ into {\em rigid} and {\em flexible} in such a way that each vertex $v$ incident with at least $\deg_H v/2$ hyperedges in $H$ is incident with at least $\deg_H v/1000$ rigid wiggly edges in $H'$ and at least half of the stubs incident with $w$ are flexible for each vertex $w$ that has stub-degree greater than $S$ in $H'$.
\end{claim}

\begin{inproof}
Let $s_v$ be the number of stubs incident with $v$ and let $d_v$ be the number of wiggly edges incident with $v$ that arise from shrinking hyperedges. Let $D'=D/400$. Observe that by construction, $d_v\geq D'$ for every $v$ (since $\edeg_{H'} v\geq D/200$ and these edges can be loops). 

Suppose that we select each shrink-pair to be rigid independently at random with probability $1/3$. 
Let $V_S$ be the set of vertices of $H'$ with stub degree at least $S$ and let $V_E$ be the set of vertices of $H'$ incident with at least $\deg_H v/2$ hyperedges in $H$. Let $S_v$ be the event that a vertex $v\in V_S$ is incident with more than $s_v/2$ rigid stubs and let $E_v$ be the event that $v\in V_E$ is incident with less than $d_v/5$ rigid edges. Then by Chernoff bound, $P(S_v)<c_S^{s_v}$ for every $v\in V_S$ and $P(E_v)<c_E^{d_v}$ for every $v\in V_E$, where $c_S, c_E \leq 0.97$. 

Now, we apply Lemma~\ref{LLL} (asymmetric Lov\'asz Local Lemma) to show that with positive probability, none of the events $S_v$, $v\in V_S$, and $E_v$, $v\in V_E$, occurs.
We need to choose $x(S_v)$ and $x(E_v)$ such that
\begin{equation}
c_S^{s_v}\leq x(S_v) \prod_{w\in A_v} (1-x(E_w)) \mbox{ for every $v\in V_S$, and} \label{eq:ALs}
\end{equation}
\begin{equation}
c_E^{d_v}\leq x(E_v) \prod_{w\in B_v} (1-x(E_w))\prod_{u\in C_v} (1-x(S_u))\mbox{ for every $v\in V_E$,}
\label{eq:ALe}
\end{equation}
where $A_v$ is the set of vertices $w\in V_E$ such that there is a shrink-pair consisting of a stub incident with $v$ and an edge incident with  $w$, $B_v$ is the set of vertices $w\in V_E$ such that there is an edge $vw$ arising from shrinking a hyperedge and $C_v$ is the set of vertices $u\in V_S$ such that there is a shrink-pair consisting of a stub incident with $u$ and an edge incident with $v$. Note that $|A_v|\leq 2s_v$ and  $|B_v|, |C_v|\leq d_v$. 
 
We set $x(E_v)=1/d_v$ and $x(S_v)=(1-c_E)/2$. Then, since $S\geq 1000$, $D'\geq 1000$, $c_S<0.97$ and $x(S_v)>0.015$,  we have $c_S< x(S_v)^{1/S}(1-1/D')^2$. It follows that if $s_v\geq S$ for every $v\in V_S$ and $d_w\geq D'$ for every $w\in V_E$, (\ref{eq:ALs}) is satisfied.
 Moreover, since $c_e<0.97$ and $D'\geq 1000$,
 \[c_E< (1/D')^{1/D'} (1-1/D')(1-(1-c_E)/2)\]
and $(1/D')^{1/D'}\leq (1/d_w)^{1/d_w}$ since $d_w\geq D'$ for every $w\in V_E$. Then, (\ref{eq:ALe}) is satisfied. 

\end{inproof}
The statement of the lemma follows.
\end{proof}

\section{Path-graphs}\label{sec:path-graphs}
In this section, we introduce the notions of {\em path-graph} on a complex graph without stubs and of {\em conflict} between complex paths. Throughout the section, we only consider complex graph without stubs. 

Let $G=(V,E_o,E_w,\emptyset,\emptyset)$ be a complex graph without stubs. A {\em path-graph $\Gamma$ {on} $G$} is a pair $(V,{\PP})$ where ${\PP}$ is a set of edge-disjoint complex paths in $G$. We denote by $\PP_{\Gamma}(v)$ the set of paths in $\PP$ with an endpoint in $v$ oriented so that $v$ is their first vertex. In particular, if $P$ is a loop-path with both ends equal to $v$, $\PP_{\Gamma}(v)$ contains both $P$ and $P^r$.

The complex graph $\underline{\Gamma}=(V,E_o',E_w',\emptyset,\emptyset)$, where $E_o'$ and $E_w'$ are the subsets of $E_o$ and $E_w$, respectively, consisting of edges of paths in ${\PP}$, is called the {\em underlying graph} of $\Gamma$. 
If $\underline{\Gamma}=G$, we say that $\Gamma$ is {\em decomposing} $G$.

We denote by $\tilde{\Gamma}$ the complex graph $(V,\emptyset,F,\emptyset,\emptyset)$ where $F$ is the multiset containing a wiggly edge $e_P=([(v_1,i_1),(v_2,i_2)],\alpha)$ for each path $P\in \PP$ where $v_1$ and $v_2$ are the first and the last vertex of $P$, respectively, $i_1,i_2$ are the initial and the terminal indices of $P$, respectively, and $\alpha$ is the length of $P$.

We now transfer some definitions for complex graphs to path-graphs. 
The {\em degree} of a vertex $v$ in $\Gamma$, denoted by $\deg_{\Gamma} v$, is the degree of $v$ in $\tilde{\Gamma}$.
We say that $\Gamma$ is {\em connected} if $\tilde{\Gamma}$ is connected, that $\Gamma$ is {\em Eulerian} if  $\tilde{\Gamma}$ is Eulerian, and that $\Gamma$ is a {\em path-tree} if $\tilde{\Gamma}$ is a tree (even if the paths of $\PP$ pairwise intersect). The {\em end set} of a path-tree $\Gamma=(V,{\PP})$  is the set $X\subseteq V$ of ends of paths in $\PP$.

We also need to speak of the length of the paths in $\PP$. Let us say that $\Gamma$ is 
a {\em $k$-path-graph} if all paths in $\PP$ are $k$-paths and that $\Gamma$ is a {\em long-path-graph} if all paths in $\PP$ are long.

For a complex path $P$ without stubs beginning with a vertex $v$, define $\sub{P}{v}$ as the maximal initial subsequence of $P$ containing no wiggly edges and at most $\ell$ ordinary edges. That is, $\sub{P}{v}$ is a path (in the usual, non-complex  sense) of length at most $\ell$. In particular, $\sub{P}{v}$ consists only of the vertex $v$ if the first edge of $P$ is wiggly.

Let $P_1$ and $P_2$ be complex paths starting with a vertex $v$. We say that $P_1$ and $P_2$ are {\em non-conflicting at $v$} if
\begin{itemize}
\item $\sub{P_1}{v}\cap \sub{P_2}{v}=\{v\}$ and
\item the initial indices of $P_1$ and $P_2$ are different or $-1$ or one of the paths is {universal at $v$}.
\end{itemize}

Otherwise, we say that $P_1$ and $P_2$ are {\em conflicting}.
Recall that $\sub{P}{v}$ is uniquely defined even for loop-paths; $\sub{P}{v}$ is the initial segment in the default orientation of $P$ and $\sub{P^r}{v}$ is the initial segment in the other orientation.

 A {\em tour} or a {\em trail} in a path-graph $\Gamma$ is an alternating sequence of vertices and paths such that they form a tour or a trail in $\tilde{\Gamma}$, respectively. A tour or a trail is called {\em non-conflicting} if every two consecutive paths $P_1, P_2$ are non-conflicting at the end they share in the tour or the trail. (We assume that loop-paths in the tour or the trail have a specified orientation.) 

Let $\Gamma$ be a path-graph and $v$ one of its vertices.
We define the {\em conflict ratio $\conf{v}$ of $v$} in $\Gamma$ as 
\[ \conf{v}=\frac{\max _{P\in \PP_{\Gamma}(v)} |\{Q\in \PP_{\Gamma}(v):Q \mbox{ is conflicting with }P\}|}{\deg_{\Gamma} v}. \]
We set $\conf{\Gamma}=\max_{v\in V(\Gamma)} \conf{v}$.
We define the {\em intersection ratio $\inter{v}$ of $v$} in $\Gamma$ as  
\[ \inter{v}=\frac{\max _{w\neq v} |\{P\in \PP_{\Gamma}(v): w\in \sub{P}{v}\}|}{\deg_{\Gamma} v}. \]
 We set $\inter{\Gamma}=\max_{v\in V(\Gamma)} \inter{v}$.
 We define the {\em index-conflict ratio $\inc{v}$ of $v$} in $\Gamma$ as 
\[\inc{v}=\frac{\max_{P\in \PP'_{\Gamma}(v)} |\{Q\in \PP_{\Gamma}(v):Q \mbox{ and } P\mbox{ have the same index at }v\}|}{\deg_{\Gamma} v}, \] where $\PP'_{\Gamma}(v)$ is the subset of $\PP_{\Gamma}(v)$ consisting of the paths which are not universal at $v$ and have index different from $-1$ at $v$. We set $\inc{\Gamma}=\max_{v\in V(\Gamma)} \inc{v}$.

The following observation describes the relation between $\inter{\Gamma}$, $\inc{\Gamma}$ and $\conf{\Gamma}$.
\begin{observation}\label{obs:inter-conf}
For every path-graph $\Gamma$, $\conf{\Gamma}\leq \inc{\Gamma} + \ell \inter{\Gamma}$.
\end{observation}

\begin{proof}
Let $P\in \PP_{\Gamma} (v)$. Since $\sub{P}{v}$ contains at most $\ell$ vertices different from $v$, the number of paths $Q\in \PP_{\Gamma} (v)$, $P\neq Q$ such that $\sub{P}{v}\cap \sub{Q}{v}\neq \{v\}$ is at most $\ell \inter{v} \deg_{\Gamma} v$. Moreover, the number of paths $Q\in \PP_{\Gamma} (v)$ conflicting with $P$ such that $\sub{P}{v}\cap \sub{Q}{v}= \{v\}$ is at most $\inc{v} \deg_{\Gamma} v$. The result follows.
\end{proof}

Observe that concatenating paths of a non-conflicting Eulerian trail in a long-path graph yields a complex path. This implies the following.

\begin{observation}\label{obs:eul-dec}
Let $G$ be an $\ell$-divisible complex graph and $\Gamma$ a long-path-graph decomposing $G$ with a non-conflicting Eulerian trail. Then, $G$ has an $\ell$-path decomposition. In particular, concatenation of the paths in the non-conflicting Eulerian trail yields a complex $\ell$-path.
\end{observation}


Jackson (cf. \cite{bib:jackson}, Theorem 6.3) proved that Eulerian path-graphs without loop-paths and with conflict ratio at most $1/2$ have non-conflicting Eulerian tours.

\begin{theorem}\label{thm:eulerian}
Every Eulerian path-graph $\Gamma$  without loop-paths with $\conf{\Gamma} \leq 1/2$
has a non-conflicting Eulerian tour.
\end{theorem}

The theorem can be modified as follows, to include path-graphs with loop-paths.

\begin{theorem}\label{thm:c-eulerian}
Every Eulerian path-graph $\Gamma$ with $\conf{\Gamma} \leq 1/4$
has a non-conflicting Eulerian tour.
\end{theorem}

\begin{proof}
The result will follow from applying Theorem~\ref{thm:eulerian} on an auxiliary path-graph $\Gamma'$ without loop-paths.
Let $\Gamma'$ be a path-graph obtained from $\Gamma$ by replacing each loop-path $P$ in $\Gamma$ with ends in $v$ by two paths $P_1$, $P_2$ between $v$ and an auxiliary vertex of degree $2$ such that $P_1$, $P_2$ are mutually non-conflicting and each of them is conflicting with all the paths with which the original path $P$ was conflicting. Thus, $\Gamma'$ has no loop-paths and the number of paths conflicting with each path increases at most by a factor of $2$, i.e.,  $\conf{\Gamma'}$ is at most $1/2$. Then, $\Gamma'$ has a non-conflicting Eulerian tour by Theorem~\ref{thm:eulerian}. Observe that by construction of $\Gamma'$, a non-conflicting Eulerian tour in $\Gamma'$ corresponds to an Eulerian tour in $\Gamma$ which is also non-conflicting.
\end{proof}

We now state results about path-graphs from~\cite{bib:orig-paths}. We need to slightly generalize them to suit our needs. However, we do not give full proofs of the modified statements as they are analogous to those in~\cite{bib:orig-paths}.
We call a complex path $P$ without stubs {\em strict} if every index appears in at most one edge of $P$.

\begin{theorem}[Theorem 3.3 in~\cite{bib:orig-paths} (generalized)]\label{thm:cones}
Let $\varepsilon$ be an arbitrarily small positive real number. For every positive integer $r$, there exists an integer $L_1$ such that if $G$ is a unit complex graph without stubs, loops and multiple edges (in other words, the underlying multigraph of $G$ is a simple graph) with minimum degree $L_1$ and $\iota(G)\leq r$, then there exists a strict-$\ell$-path-graph $\Gamma$ on $G$ with $\inter{\Gamma}\leq \varepsilon$, $\deg_{\Gamma} v/\deg_G v \in [\frac{1-\varepsilon}{\ell},\frac{1+\varepsilon}{\ell}]$ and $\deg_{G\setminus E(\underline{\Gamma})} v \leq \varepsilon \deg_{\Gamma} v$ for every vertex $v$ of $G$. 
\end{theorem}

\begin{proof}
The proof is analogous to the proof of Theorem 3.3 in~\cite{bib:orig-paths}, except that we need to ensure that every index appears at most once in each path. This can be done in the same way as ensuring that every vertex appears at most once in each path.
\end{proof}

Next, we state a variant of Theorem 3.4 from~\cite{bib:orig-paths}.
\begin{theorem}\label{thm:low-conf-simple} 
For any positive integers $c$ and $r$, there exists an integer $L_2$ such that if $G$ is a unit complex graph  without stubs, loops and multiple edges with minimum degree $L_2$ and $\iota(G)\leq r$, there exists a strict-long-path-graph $\Gamma$ decomposing $G$ with $\inter{\Gamma}\leq 1/c$ and $\deg_{\Gamma} v> \deg_G v/2\ell$ for every vertex $v$ of $G$. 
\end{theorem}

\begin{proof}
The proof is analogous to the proof of Theorem~3.4 in~\cite{bib:orig-paths}, using Theorem~\ref{thm:cones} instead of Theorem 3.3 in~\cite{bib:orig-paths}. To obtain $\inter{\Gamma}\leq 1/c$ instead of $\inter{\Gamma}\leq 1/4$, it is sufficient to slightly modify constants in the proof. In particular, to choose $G_1$ to be some $1/(2c+1)$-fraction of $G$.
\end{proof}

Note that Theorem~\ref{thm:low-conf-simple} is stated only for unit complex graphs with simple underlying graphs. Therefore, we cannot apply it directly. Nevertheless, we can use it to prove a slightly weaker variant for all complex graphs.

\begin{theorem}\label{thm:low-conf-complex}
For any positive integers $c$ and $r$, there exists an integer $L$ such that for every complex graph $G$ without stubs, with minimum degree at least $L$ and $\iota(G)\leq r$, there is a long-path-graph $\Gamma$ decomposing $G$ with $\inter{\Gamma}\leq 1/c$ and $\deg_{\Gamma} v> \deg_G v/2\ell$ for all vertices $v$.
\end{theorem}

\begin{proof}
Let $L$ be greater than $L_2$ obtained from Theorem~\ref{thm:low-conf-simple} applied to $r$ and $c$ as in our hypothesis. Let $G^*$ be the complex graph obtained from $G$ by replacing each wiggly edge $e$ with endpoints $v_1$,$v_2$ by a gadget $H_{e}$ consisting of a clique of size $L+1$ and unit wiggly edges $e_1$, $e_2$ with endpoints $v_1,v_2'$ and $v_1,'v_2$, respectively, where $v_1'$ and $v_2'$ are (arbitrary) distinct vertices of the clique. The edges $e_1$ and $e_2$ have the same indices as $e$, i.e., they have index $i^e_1$ at $v_1$ and $v_1'$ and index $i^e_2$ at $v_2$ and $v_2'$. All the remaining edges of the gadget are ordinary. See Figure~\ref{fig:gadget} for an illustration.

\begin{figure}
\begin{center}
\includegraphics[scale=1]{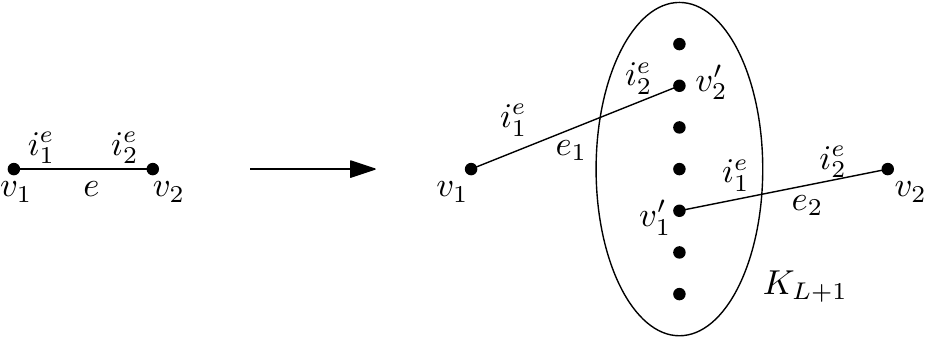}
\caption{Replacing an edge $e$ by a gadget $H_{e}$.}
\label{fig:gadget}
\end{center}
\end{figure}

Note that the resulting complex graph $G^*$ is unit, has minimum degree at least $L$ and the underlying graph of $G^*$ is simple. By applying Theorem~\ref{thm:low-conf-simple} with $c$ and $r$, we obtain a strict-long-path-graph $\Gamma^*$ decomposing $G^*$ with $\inter{\Gamma^*}\leq 1/c$ and $\deg_{\Gamma^*} v > \deg_{G^*} v/2\ell$ for every vertex $v$ of $G^*$. 

Let $H_e$ be the gadget corresponding to the wiggly edge $e$ with endpoints $v_1$ and $v_2$. Since $e_1$ and $e_2$ have the same indices, they cannot occur in the same strict path of $\Gamma^*$. Thus, there are two different paths $P_1$ and $P_2$ in $\Gamma^*$ such that $e_1\in P_1$ and $e_2\in P_2$. Moreover, since $e_1$ and $e_2$ form a $2$-edge-cut in $G^*$, $P_1$ and $P_2$ have exactly one endpoint in $H_e\setminus \{v_1,v_2\}$ each. Let $P'_1= P_1\setminus (H_e\setminus \{v_1,v_2\})$, that is, $v_1$ is an end of $P'_1$. We define $P'_2$ analogously. Then, $P_e=P'_1\ct e \ct P'_2$ is a long path. 

If we replace $P_1$, $P_2$ in $\Gamma^*$ by $P_e$ for every gadget $H_e$, we obtain the long-path-graph $\Gamma$ which is a decomposition of $G$. Moreover, $\deg_{\Gamma} v=\deg_{\Gamma^*} v$ and $\inter{v}$ is the same in $\Gamma$ and $\Gamma^*$ for every $v\in V(\Gamma)$. It follows that $\Gamma$ satisfies the statement of the lemma.
\end{proof}

Note that a path-graph resulting from Theorem~\ref{thm:low-conf-complex} is not guaranteed to be connected.

Now we introduce the notion of {\em rainbow structures} in complex graphs. These structures are later used for constructing collections or edge-disjoint paths with specified properties.

Let $G$ be a complex graph without stubs. A {\em rainbow $(k,m,\delta,\varepsilon)$-structure} $\GG$ in $G$ is a collection of $k$ pairwise edge-disjoint spanning subgraphs $G_1,\ldots, G_k$ of $G$ such that
\begin{itemize}
\item $\deg_{G_1} v \geq 2(\delta+\varepsilon) \deg_{G} v+2m+1$ and 
\item $\deg_{G_{i+1}} v\geq 2\varepsilon \deg_{G} v + \deg_{G_{i}} v+2i+2m$
\end{itemize}
for every $v\in V(G)$ and $i\in [k-1]$. Abusing notation, we will not always distinguish between $\GG$ and $\bigcup_{i\in [k]} G_i$.

\begin{observation}\label{obs:e-rainbow}
For every positive integers $k$, $m$ and positive reals $\delta,\varepsilon\leq 2^{-(k+3)}$, every complex graph $G$ without stubs and with minimum degree at least $2^{k+2}(k+m+3)$ has a $(k,m,\delta,\varepsilon)$-rainbow structure.
\end{observation}

\begin{proof}
Let $G_1$, $G_2$, $G_3,\ldots, G_{k}$ be pairwise edge-disjoint subgraphs of $G$ such that $G_i$ is a $2^{-(k-i+1)}$ fraction for every $i\in [k]$.
Such fractions exist by Lemma~\ref{lem:fraction}. 
 We show that $G_1,\ldots, G_k$ form a $(k,m,\delta,\varepsilon)$-rainbow structure. We have that $\deg_{G_{i+1}} v \geq 2 \deg_{G_i} v -4$ for every $i\in [k-1]$ and $\deg_{G_1} v \geq \deg_{G} v/2^{k} -2$ for every $v$. (Thus, in particular, if $\deg_{G_1} v \geq 4$, $\deg_{G_i} v\geq \deg_{G_1} v$ for every $i$.)

Since the degree of every vertex $v$ is at least $2^{k+2}(k+m+3)$, 
\[\frac{\deg_{G_1} v}{2} \geq \frac{2^{k+2}(k+m+3)}{2^{k+1}}-1\geq 2k + 2m + 5.\] 
Note that $2(\delta+\varepsilon)\deg_G v \leq \deg_G v/2^{k+1}\leq \deg_{G_1} v /2+1$. 
Thus, $\deg_{G_1} v \geq 2(\delta+\varepsilon)\deg_G v + 2k + 2m + 4$. Consequently, 
\begin{align*}
\deg_{G_{i+1}} v &\geq 2 \deg_{G_i} v -4\\
&\geq \deg_{G_i} v +\deg_{G_1} v -4\\
&\geq \deg_{G_i} v + 2\varepsilon\deg_G v + 2k + 2m
\end{align*}
for every $i\in [k-1]$.
\end{proof}

Next, we show how a $(k,m,\delta,\varepsilon)$-rainbow structure can be used for constructing edge-disjoint complex paths. For that, we need the following observation about the existence of an {\em almost balanced orientation} of a  complex graph. An {\em orientation} of a complex graph without stubs is defined analogously to an orientation of multigraphs, i.e., it is an assignment of a direction (that is, an ordering of the vertices) to each edge. For a given orientation $D$, the {\em indegree} and the {\em outdegree} of a vertex $v$, denoted by $\deg_D^- v$, $\deg_D^+ v$, respectively, is the number of edges incident with $v$ that are directed towards $v$ or away from $v$ in $D$, respectively.

\begin{observation}\label{obs:balor}
Every complex graph $G$ without stubs has an orientation such that $|\deg^{-} v-\deg^{+} v|\leq 1$ for every vertex $v$. We call such orientation {\em almost balanced}.
\end{observation}

\begin{proof}
Let $G'$ be a complex graph obtained from $G$ by adding a vertex $v$ which is adjacent to all vertices of odd degree in $G$. Since the number of vertices of odd degree in $G$ is even, all vertices have even degree in $G'$. Let $D'$ be an Eulerian orientation of $G'$ (that is, we consider an Eulerian tour for each connected component of $G'$ and orient the edges according to these Eulerian tours). Consider the orientation $D$ of the edges of $G$ as in $D'$. Since $\deg^{-}_{D'} v=\deg^{+}_{D'} v$ for every vertex $v$, we have $|\deg^{-}_{D} v-\deg^{+}_{D} v|\leq 1$.
\end{proof}

\begin{lemma}\label{lem:rainbow}
Let $q$ be a positive integer and let $G$ be a $\varepsilon/q$-free complex graph without stubs. Let $\GG$ be a $(k,m,\delta,\varepsilon)$-rainbow structure in $G$. For every $v$ in $G$, let $n_v=\lceil\delta \deg_G v\rceil$, let $M^v_i$ be a set of vertices different from $v$ of size at most $m$ and let $\JJ^v_i$ be a set of indices different from $-1$ with $|\JJ^v_i|\leq q$ for every $i\in[n_v]$.

Then, for every $v$ of $G$, one can construct a collection of $n_v$ edge-disjoint complex paths $\PP_v=\{P^v_1,\ldots, P^v_{n_v}\}$ in $\GG$ starting in $v$, each containing $k$ edges, such that $\sub{P^v_i}{v}$ avoids vertices of $M^v_{i}$ and the initial index of $P^v_i$ is not in $\JJ^v_i$. Moreover, the sets $E(\PP_v)$ are edge-disjoint and the paths can be constructed one by one in an online manner. That is, we can construct the paths in any prescribed order, without considering any subsequent paths (in particular, without knowledge of $M^v_{i}$ and $\JJ^v_i$ for any subsequent path $P^v_i$). 
\end{lemma}

\begin{proof}
By Observation~\ref{obs:balor}, there exists an almost balanced orientation $D_j$ of $G_j$ for every $j\in [k]$. We construct a path $P^v_i$ such that its $j$-th edge belongs to $G_j$ and is oriented from the $j$-th to the $j+1$-th vertex of $P^v_i$ in $D_j$. 
We start with $P^v_j$ consisting only of the vertex $v$ and repeatedly extend it.

Since $v$ has outdegree at least $(\delta + \varepsilon) \deg_{G} v+m$ in $D_1$, after constructing less than $\delta \deg_G v$ paths starting from $v$, there are at least $\varepsilon \deg_{G} v + m+1$ edges leaving $v$ in $D_1$ which were not used by the previous paths. Since $G$ is $\varepsilon/q$-free, at most $\varepsilon \deg_G v$ of these edges have index in $\JJ^v_i$ at $v$. If there is a wiggly edge among the (at least) $m+1$ remaining edges, it can be used to extend $P^v_i$, trivially satisfying that $\sub{P^v_i}{v}\cap M^v_i=\emptyset$. Otherwise, since the ordinary edges form a simple graph, there is an ordinary edge incident with $v$ such that its other endpoint does not belong to $M^v_i$ and therefore it can be used to extend $P^v_i$.

By an analogous reasoning we get that the path $P^v_i$ of length $j$ with last vertex $u$ can be extended by an edge leaving $u$ in $D_{j+1}$ that was not used by any of the previous paths. It is enough to notice that $\deg^-_{D_{j}} u+ j+m+\varepsilon \deg_G u \leq \deg^+_{D_{j+1}} u$. Then, out of $\deg^+_{D_{j+1}}  u$ edges leaving $u$, at most $\deg^-_{D_{j}} u-1$ were used by the previously constructed paths, leaving at least $\varepsilon\deg_G u+j+m+1$ possible edges for extending $P^v_i$. At most $\varepsilon \deg_G u$ of these edges have the same index as the terminal index of $P^v_i$  (unless it is $-1$), or have index in $\JJ^v_j$ if $P^v_i$ contains only ordinary edges. Then, as before, either one of at least $j+m+1$ remaining edges incident with $v$ is wiggly, or its other endpoint is disjoint from $M^v_i$ and from all the vertices already contained in $P^v_i$. Such an edge can be used for extending $P_i^v$.

\end{proof}

Later, we will also use Lemma~\ref{lem:rainbow} for constructing paths shorter than what is specified by the parameter of the structure. Observe that we can simply take only the initial subpaths of the paths constructed in the lemma.

In the following two lemmas, we describe construction of spanning $1$-path-trees and $\ell$-path-trees.

\begin{lemma}\label{lem:one-tree}
There exists $\varepsilon_1>0$ such that for any $\varepsilon_1$-free unit complex graph $G$ without stubs and any complex spanning tree $T$ on the vertex set $V(G)$ (that is, the underlying graph of $T$ is a spanning tree on $V(G)$), edge-disjoint from $G$ and such that $\deg_T v\leq \varepsilon_1 \deg_G v$ for every $v$, there exists a spanning $1$-path-tree on $T\cup G$.
\end{lemma}

\begin{proof}
Let $\varepsilon_1=\min(\frac{1}{2^{\ell+1}(\ell+2)},2^{-(\ell+2)})$. Then, by Observation~\ref{obs:e-rainbow}, there exists an $(\ell-1,0,\varepsilon_1,\varepsilon_1)$-rainbow structure $\GG$ in $G$.

We define a {\em structured tree $T'$ over $G\cup T$} to be a tree whose vertices are subsets $X_i$ of $V(G)$ and which satisfies that 
\begin{itemize}
\item the subsets $X_i$ form a partition of $V(G)$,
\item every $X_i$ is spanned by a $1$-path-tree $\Pi_{X_i}$ on $G\cup T$,
\item each edge between two vertices $X$ and $Y$ in $T'$ corresponds to an edge in $T$ with one endpoint in $X$ and the other in $Y$, and
\item every edge of $G\cup T$ is used at most once. That is, it either corresponds to an edge between the two vertices of $T'$ or it is involved in at most one of the $1$-path trees.
\end{itemize}

Let $T'$ be the structured tree such that its vertices are all the maximal $1$-subtrees of $T$. (By maximality of $1$-subtrees, every edge in $T'$ corresponds to a non-unit edge in $T$.)
We repeatedly modify $T'$ using the following operation, which decreases the number of vertices of $T'$.

Choose an edge $X_uX_v$ of $T'$ such that $X_u$ is a leaf. Let $e$ be the edge of $T$ corresponding to $X_uX_v$, let $u\in X_u$, $v\in X_v$ be the endpoints of $e$ and $j$ its index at $v$.  
By Lemma~\ref{lem:rainbow}, there exists a complex path $Q$ in $\GG$ starting in $v$ of length congruent to $1-\alpha(e)$ modulo $\ell$, with the initial index different from $j$ or equal to $-1$, edge-disjoint from all previously chosen paths in $\GG$.

Let $y$ be the last vertex of $Q$ and let $Y$ be the vertex of $T'$ such that $y\in Y$. If $Y\neq X_u$,  we remove $X_u$ and replace $Y$ by $X_u\cup Y$ in $T'$. This yields a structured tree, because $e\circ Q$ is a complex $1$-path and therefore, $\Pi_{X_u}\cup \Pi_{Y}\cup (e\circ Q)$ is a $1$-path tree spanning $X_u\cup Y$. 

If $Y=X_u$, there is an edge $e'$ in $Q$ with endpoints $u'\in X_u$ and $v'\in X_{v'}$ such that  $X_u\neq X_{v'}$. In this case, we remove $X_u$ and replace $X_{v'}$ by $X_u\cup X_{v'}$ in $T'$.

Recall that $e'$ is a unit edge, therefore $\Pi_{X_u}\cup \Pi_{X_{v'}}\cup e'$ is a $1$-path-tree and thus, $T'$ is a structured tree.

We repeat this process until we obtain a structured tree $T'$ over $G\cup T$ that has only one vertex $X$. Then, $\Pi_X$ is the desired $1$-path-tree.
\end{proof}

\begin{lemma}\label{lem:rainbow-tree}
Assume that $\ell$ is even. Then, there exists $\varepsilon_2>0$ such that if $G$ is an $\varepsilon_2/2$-free bipartite unit graph without stubs with parts $A,B$, and $\Phi, \Psi$ are edge-disjoint $1$-path-trees both with end set $V(G)$ (i.e., $\tilde{\Phi}$ and $\tilde{\Psi}$ are spanning trees on $V(G)$ and $\underline{\Phi}$ and $\underline{\Psi}$ are edge-disjoint) edge-disjoint from $G$ with bipartition $A,B$, such that $\deg_{\Phi} v\leq \varepsilon_2\deg_G v$ and $\deg_{\Psi} v\leq \varepsilon_2\deg_G v$ for every $v$, there exist edge-disjoint $\ell$-path-trees $\Pi_A$ and $\Pi_B$ in $\underline{\Phi}\cup \underline{\Psi}\cup G$, such that $\Pi_A$ is spanning $A$ and $\Pi_B$ is spanning $B$.

\end{lemma}

\begin{proof}
Let $\varepsilon_2=\min(\frac{1}{2^{\ell+1}(3\ell+2)},2^{-(\ell+2)})$. By Lemma~\ref{obs:e-rainbow}, there exists an $(\ell-1,2\ell,\varepsilon_2,\varepsilon_2)$-rainbow structure $\GG$ in $G$.

We start by building $\Pi_A$ using $\Phi$ and $\GG$. It is important to note that we construct only paths in $\GG$ starting in vertices of $B$. Therefore, we can later build $\Pi_B$ using paths in $\GG$ starting in vertices of $A$ in such a way that $\Pi_A$ and $\Pi_B$ are edge-disjoint.

We define a {\em structured tree} $T$ over $G$ and $\Phi$ as a rooted tree whose vertices are subsets $X_i$ of $V(G)$ and which satisfies that 
\begin{itemize}
\item the subsets $X_i$ are pairwise vertex-disjoint and $A\subseteq\bigcup_{X_i\in V(T)} X_i$,
\item the root $X_r$ of $T$ is a subset of $A$
\item every vertex of $T$ containing a vertex in $B$ is a singleton,
\item every leaf $X_i$ of $T$ is spanned by an $\ell$-path-tree $\Pi_{X_i}$,
\item every edge $X_iX_j$ in $T$ corresponds to a path of $\Phi$ with endpoints $x_i\in X_i$ and $x_j\in X_j$, and
\item every edge of $G\cup \underline{\Phi}$ is involved in at most one path of $T$ or of some $\Pi_{X_i}$.
\end{itemize}

We start with $T=\tilde{\Phi}$ (all vertices of $T$ being singletons) and repeatedly perform the following operations which preserve the structured tree conditions and decrease the number of vertices of $T$. If $T$ has a leaf containing a vertex of $B$, we delete it. Otherwise, let $Y$ be an internal vertex of $T$ with maximum distance from the root. Then, $Y=\{y\}$ with $y\in B$, in particular, $Y$ is not the root.

\begin{itemize}
\item If $Y$ has at least two children, let  $X_1$ and $X_2$ be two of them. Let $P_1$ and $P_2$ be the $1$-paths of $\Phi$ corresponding to the edges $X_1Y$ and $X_2Y$ in $T$. 
By Lemma~\ref{lem:rainbow} for $q=2$, there exists a path $P$ in $\GG$ of length $\ell-1$ starting in $y$ which is non-conflicting with $P_1$ and $P_2$ at~$y$ (since $\sub{P_1}{y}$ and $\sub{P_2}{y}$ have at most $\ell$ vertices each and $P$ can avoid the indices of $P_1$ and $P_2$ at~$y$) and edge-disjoint from all previously constructed paths in $\GG$. 
Let $z$ be the last vertex of $P$. 
By the parity of $\ell$, we have $z\in A$. Let $Z$ be the vertex of $T$ such that $z\in Z$. Without loss of generality (with possible exchange of $X_1$ and $X_2$) we might assume that $Z\neq X_1$. Note that $P_1\circ P$ is a complex $\ell$-path. 
Therefore, we can replace the vertex $Z$ by $Z\cup X_1$, because $\Pi_Z\cup \Pi_{X_1}\cup (P_1\circ P)$ is an $\ell$-path-tree. This decreases the number of children of $Y$ and the number of vertices of $T$.

\item If $Y$ has a single child $X_1$, we consider the parent $X_2$ of $Y$ and let $P_1$ and $P_2$ be the $1$-paths of $\Phi$ corresponding to the edges $X_1Y$ and $X_2Y$ in $T$. Then again, by Lemma~\ref{lem:rainbow} for $q=2$, there exists  a path $P$ in $\GG$ of length $\ell-1$ starting at $y$ which is non-conflicting with $P_1$ and $P_2$ and edge-disjoint from all previously constructed paths in $\GG$. Let $z$ be the last vertex of $P$. If $z\notin X_1$, we replace $Z$ by $Z\cup X_1$ as in the previous case. Otherwise, we observe that $P_2\circ P$ is an $\ell$-path from $X_1$ to $X_2$. Therefore, we can replace $X_2$ by $X_2\cup X_1$, spanned by the $\ell$-path-tree $\Pi_{X_1}\cup \Pi_{X_2}\cup (P_2\circ P)$. Again, this decreases the number of children of $Y$ and the number of vertices of $T$.

\end{itemize}

Eventually, we obtain $T$ with only one vertex $X$. Then $\Pi_X$ is an $\ell$-path tree spanning $A$. We let $\Pi_A:=\Pi_X$.
We build $\Pi_B$ in an analogous way by defining a structured tree over $G$ and $\Psi$ with roles of $A$ and $B$ swapped. As we argued above, it is possible to construct $\Pi_B$ which is edge-disjoint from $\Pi_A$.

\end{proof}

\section{Proof of a key lemma}\label{sec:core}
This section is devoted to proving the following lemma, which plays a key role in the proof of Theorem~\ref{thm:main}.

\begin{lemma}\label{lem:core}
For every integer $\ell$ and $\beta\in (0,1]$, there exists $k$ such that for every $r$ and $s$ there exists $d$ such that every complex graph $G_0$ which

\begin{itemize}
\item is $k$-edge-connected,
\item is $\ell$-divisible,
\item has minimum edge-degree at least $d$,  
\item has minimum unit edge-degree at least $\beta d$,
\item has maximum stub-degree at most $s$, and
\item $\iota(G_0)\leq r$
\end{itemize}
can be decomposed into $\ell$-paths such that every path contains at most one stub.
\end{lemma}

Before proving the lemma, we would like to recall some facts that are used in the proof. First, note that if $(A,B)$ is a maximum cut in a $2k$-edge-connected multigraph $G$, the multigraph $G'$ obtained from $G$ by deleting edges with both ends in $A$ or in $B$ is $k$-edge-connected. Moreover, if $G$ has at least two vertices and no loops, $\deg_{G'} v\geq \deg_G v/2$ for every $v\in V(G)$.

Second, for every tree $T$ and every set $U\subseteq V(T)$ of even size, it is possible to find a subgraph $T'$ of $T$ such that $U$ is the set of vertices of odd degree in $T'$.

Informally, the proof of Lemma~\ref{lem:core} proceeds as follows. First, we construct a collection of $\ell$-paths containing all the stubs so that the rest of the graph is still of high minimum unit degree and high connectivity. Next, we condense some paths to make every vertex incident with many unit non-loops. We take a highly connected bipartite subgraph of this condensation and construct $\ell$-path-trees spanning its parts. Finally, we find a long-path-decomposition of the rest of the graph with low conflict and make it connected and (almost) Eulerian using some $\ell$-paths from the $\ell$-path trees. We then use Theorem~\ref{thm:c-eulerian} to show that this long-path-graph has a non-conflicting Eulerian tour or trail, which yields an $\ell$-path-decomposition.

\begin{proof}[Proof of Lemma~\ref{lem:core}]
We assume that $\ell$ is even. If not, we apply Lemma~\ref{lem:core} for $2\ell$; there exist $k'$ and $d'$ such  that if  $G_0$ is not $2\ell$-divisible but satisfies all the remaining requirements of Lemma~\ref{lem:core} for $2\ell$, it is possible to greedily construct a path $P$ in $G_0$ consisting of $\ell$ unit edges such that $G_0\setminus P$ satisfies the requirements of Lemma~\ref{lem:core} for $2\ell$.

Let $\varepsilon=2^{-c}$ for $c\in \NN$ such that $\varepsilon\leq \min(\frac{\beta}{32(\ell-1)+1}, \varepsilon_1, \varepsilon_2, 1/1000\ell^2)$  where $\varepsilon_1$ is as in Lemma~\ref{lem:one-tree} and $\varepsilon_2$  is as in Lemma~\ref{lem:rainbow-tree}.
Let $t= 2 \max(L,2^n)$ for $n$ and $L$ as in Lemma~\ref{cor:sparse-tree} for $\varepsilon^3/4$ and $m=2$. Let $k=2^n$ where $n$ is as in Lemma~\ref{cor:sparse-tree} for $\varepsilon$ and $m=\lceil \log t\rceil$.

Now, we construct a collection of paths containing all stubs of $G_0$ such that the rest of the graph is highly connected and of high minimum degree. Let $S$ be the set of stubs in $G_0$ and let $G_1=G_0\setminus S$.
Note that $G_1$ is $k$-edge-connected with minimum degree at least $d$. By Lemma~\ref{cor:sparse-tree} and the choice of $k$, $G_1$ contains $t$ edge-disjoint spanning trees $T_1,\ldots, T_t$ such that $\sum_{i\in[t]}\deg_{T_i} v\leq\varepsilon \deg_{G_1} v$ for every $v$, if $d$ is sufficiently large.

Let $H$ be the unit subgraph of $G_1\setminus \left(\cup_{i\in [t]} E(T_i)\right)$ containing all its unit edges. Note that $H$ has degree at least $(\beta-\varepsilon) d$. 
Let $F_0, F_1$ be edge-disjoint $1/4$-fractions of $H$. Such $F_0, F_1$ exist by Lemma~\ref{lem:fraction}. 

Let $d_{F_0}$ be the minimum degree in $F_0$. Note that $d_{F_0}\geq (\beta-\varepsilon)d/4-2$.  Let $\gamma_1=s/d_{F_0}$ and $\gamma_2=r/d_{F_0}$. If $d$ is sufficiently large, $F_0$ satisfies the assumptions of Observation~\ref{obs:e-rainbow} for $\ell-1$, $0$, $\gamma_1$ and $\gamma_2$. That is, $d_{F_0}\geq 2^{\ell+1}(\ell+2)$ and $\gamma_1,\gamma_2<2^{-(\ell+2)}$.
Thus, there exists an $(\ell-1,0,\gamma_1,\gamma_2)$-rainbow structure $\FF$ in $F_0$.

For each stub $s\in S$, we construct a path $P$ in $\FF$ such that $s\ct P$ is an $\ell$-path. This is possible by Lemma~\ref{lem:rainbow} (with $q=1$), since $F_0$ is $\gamma_2$-free and $s\leq \gamma_1\deg v$ for every $v$.
Let $\PP$ be the collection of these $\ell$-paths. Clearly, each path in $\PP$ has exactly one stub and $G_1\setminus E(\PP)$ has no stubs. Thus, if $G_1\setminus E(\PP)$ has an $\ell$-path-decomposition $\QQ$, $\PP\cup\QQ$ forms an $\ell$-path-decomposition of $G_0$ such that each of its paths contains at most one stub.

Next, we show how to construct an $\ell$-path-decomposition of $G_1\setminus E(\PP)$.
If $G_1$ has only one vertex, the existence of an $\ell$-path-decomposition of $G_1\setminus E(\PP)$ follows from Lemma~\ref{lem:singleton}, because for $d$ sufficiently large, $G_1\setminus E(\PP)$ has more than $r$ loops and a path obtained by concatenating all the loops is an $\ell$-path, since $G_1\setminus E(\PP)$ is $\ell$-divisible.

From now on we assume that $G_1$ has at least two vertices.

Observe that for $d$ sufficiently large, $F=F_1\cup \left(\bigcup_{i\in [t]} T_i\right)$ satisfies the requirements of Lemma~\ref{lem:condens}. In particular, 
\begin{align*}
1-\udeg_{F} v/\deg_{F}v &\leq \frac{\varepsilon \deg_{G_1} v}{\frac{\beta-\varepsilon}{4} \deg_{G_1} v -2+t}\\
&\leq  \frac{\varepsilon\deg_{G_1} v}{\frac{\beta-\varepsilon}{4} \deg_{G_1} v}\leq \frac{4\varepsilon}{\beta-\varepsilon}
\end{align*}
for every $v$, since $t\geq 2$. Thus, by the choice of $\varepsilon$, $\udeg_{F} v/\deg_{F}v\geq 1-\frac{1}{8(\ell-1)}$ for every $v$. 

Let $G_2$ be the graph obtained from $F$ by applying Lemma~\ref{lem:condens}. Then, $G_2$ is $t$-edge-connected, has no loops and every vertex satisfies either~\ref{it:big-vert} or~\ref{it:univ-vert} of Lemma~\ref{lem:condens}, in particular, $\udeg_{G_2} v\geq 3/4 \deg_{G_2} v$ for every $v$. Let $(A,B)$ be a maximum cut in $G_2$ and let $G_2'$ be the bipartite graph with parts $A$ and $B$ obtained from $G_2$ by removing all the edges inside $A$ and $B$. Observe that $G_2'$ is $t/2$-edge-connected (and has minimal degree at least $t/2$) and at least half of the edges incident with $v$ in $G_2'$ are unit for every $v$. Thus, by Lemma~\ref{cor:sparse-tree} and by the choice of $t$, there exist edge-disjoint spanning trees $T'_1,T'_2,T'_3, T'_{4}$ such that $\sum_{i\in[4]}\deg_{T'_i} v<\frac{\varepsilon^3}{4} \deg_{G_2'} v$ for every $v$. Note that this also implies that $\varepsilon^3t/8\geq 4$, since every vertex has degree at least one in every $T_i, i\in [4]$.
Let $H'$ be the unit subgraph of $G_2'\setminus \left(\cup_{i\in [4]} E(T'_i)\right)$ containing all unit edges and note that $\sum_{i\in[4]}\deg_{T'_i} v\leq\varepsilon^3 \deg_{H'} v$ since $\deg_{H'} v\geq \deg_{G_2'} v/2- \frac{\varepsilon^3}{4} \deg_{G_2'} v\geq \deg_{G_2'} v/4$.

Let $F_1',F_2',F_3', F_4'$ be $\varepsilon^2$-fractions and $F_5',F_6'$ be $2\varepsilon$-fractions of $H'$, all of them mutually edge-disjoint. Their existence follows from Lemma~\ref{lem:fraction}. Note that if $d$ is sufficiently large, $F_1',\ldots, F_4'$ are $\varepsilon$-free because every vertex in $F_i'$, $i\in [4]$, either has degree at least $r/\varepsilon$ (vertices satisfying~\ref{it:big-vert} in $G_2$) or has degree at least $\varepsilon^2t/8-2>0$ and is incident only with ordinary edges and edges that have index $-1$ at it (vertices satisfying~\ref{it:univ-vert} in $G_2$). 

Moreover, $\deg_{T_i'} v \leq \varepsilon \deg_{F_i'} v$ for every $v$ and $i\in [4]$, since $\deg_{T_i'} v\leq \varepsilon^3 \deg_{H'} v-3$ (since $v$ has degree at least one in each tree $T'_j, j\in [4]\setminus\{i\}$). 
Thus, by  Lemma~\ref{lem:one-tree}, there exist $1$-path spanning trees $\Phi_1,\ldots, \Phi_{4}$ such that $\underline{\Phi_i}\subseteq T_i'\cup F_i'$, $i\in [4]$. Next, from the same argument as above, it follows that $F_5'$ is $\varepsilon/2$-free for $d$ sufficiently large.
Moreover, $\deg_{\Phi_1} v \leq \deg_{T_1'}v+\deg_{F_1'}v\leq \varepsilon \deg_{F_5'} v$ and similarly, $\deg_{\Phi_2} v \leq \varepsilon \deg_{F_5'} v$ for every $v$. Therefore, by Lemma~\ref{lem:rainbow-tree}, there exist $\ell$-path-trees $\Pi_A$, $\Pi_B$ on $\underline{\Phi_1}\cup\underline{\Phi_2}\cup F'_5$ spanning $A$ and $B$. Analogously, there are $\ell$-path trees $\Pi'_A$, $\Pi'_B$ spanning $A$ and $B$ contained in $\underline{\Phi_3}\cup\underline{\Phi_4}\cup F'_6$.
 
Let $Q$ be a long path in $H'\setminus E(\bigcup_{i\in[6]}F_i')$ with one end in $A$ and the other end in $B$. Such $Q$ exists, because $H'\setminus E(\bigcup_{i\in[6]}F_i')$ is $1/2$-free and every vertex has degree at least $2\ell+1$ (by the same arguments as above), so the path can be built greedily. 

We define the path-graph $\Upsilon =\Pi_A \cup \Pi_B\cup \Pi'_A\cup \Pi'_B \cup Q$. 
Observe that 
\begin{align*}
\deg_{\Upsilon} v\leq \deg_{\underline{\Upsilon}} v&\leq \sum_{i=1}^{4}\deg_{\underline{\Phi_i}} v + \deg_{F'_5}v +\deg_{F'_6} v +2\\
&\leq \sum_{i=1}^{4}\deg_{F'_i} v +  \sum_{i=1}^{4}\deg_{T'_i} v + 2( 2\varepsilon \deg_{H'} v +2)+2\\ 
&\leq 4(\varepsilon^2 \deg_{H'} v +2) + \varepsilon^3 \deg_{H'} v+ 2( 2\varepsilon \deg_{H'} v +2)+2\\
&\leq 9\varepsilon \deg_{G_1} v +14
\end{align*}
for every $v$, since $\deg_{H'} v\leq \deg_{G_1} v$.
Let $G_3=\left(G_1\setminus E(\PP\cup F)\right) \cup \left(G_2\setminus E(\underline{\Upsilon})\right)$. Note that $G_3\cup \underline{\Upsilon}$ is a condensation of $G_1\setminus E(\PP)$. Thus, the existence of an $\ell$-path-decomposition of $G_3\cup \underline{\Upsilon}$ implies the existence of an $\ell$-path-decomposition in $G_1\setminus E(\PP)$.

Observe that $\deg_{G_3} v\geq \deg_{G_1} v/4$ for every $v$ for $d$ sufficiently large, because 
\begin{align*}
\deg_{G_3} v&\geq \deg_{G_1} v-\deg_{F_1} v-\deg_{F_0} v-\sum_{i\in [t]} \deg_{T_i} v\\
&\geq \deg_{G_1} v-2(\deg_{G_1} v/4+2)-\varepsilon \deg_{G_1} v=(1/2-\varepsilon)\deg_{G_1}v-4.
\end{align*}

Moreover, if $d$ is sufficiently large, applying Theorem~\ref{thm:low-conf-complex} to $G_3$ yields a long-path-graph $\Lambda$ decomposing $G_3$ with $\inter{\Lambda}<1/12\ell$ and $\deg_{\Lambda}v>\deg_{G_3}v/2\ell$.

Let $\Lambda'$ be the long-path-graph obtained from $\Lambda\cup \Pi_A\cup \Pi_B\cup Q$ by adding some of the $\ell$-paths of $\Pi'_A$ and $\Pi'_B$ such that at most one vertex in $A$ and one vertex in $B$ have odd degree. Let $\LL$ be the set of the remaining $\ell$-paths of $\Pi'_A$ and $\Pi'_B$. Note that $\Lambda'$ is connected, since $\Pi_A\cup \Pi_B\cup Q$ is connected. We claim that $\conf{\Lambda'}<1/4$.

\begin{claim}
$\conf{\Lambda'}<1/4$.
\end{claim}
\begin{inproof}
Using Observation~\ref{obs:inter-conf}, we have 
\begin{align*}
\conf{\Lambda'} &\leq \inc{\Lambda'} + \ell \inter{\Lambda'}\\
&\leq \inc{\Lambda'}+\ell \inter{\Lambda} + \ell \max_{v}\left(\frac{\deg_{\Upsilon} v}{\deg_{\Lambda} v}\right).
\end{align*} 

We argue that $\inc{\Lambda'}\leq 1/12$ and $\max_{v}\left(\frac{\deg_{\Upsilon} v}{\deg_{\Lambda} v}\right)\leq 1/12\ell$. Together with $\inter{\Lambda}< 1/12\ell$, it yields the result.

Note that $\deg_{\Lambda'} v \geq \deg_{\Lambda} v \geq \deg_{G_3} v/2\ell\geq d/8\ell$ and $\inc{\Lambda'}\leq r/\deg_{\Lambda'} v$. Thus, for $d$ sufficiently large, $\inc{\Lambda'}\leq 8\ell r/d\leq 1/12$.

By the choice of $\varepsilon$, we have $\deg_{\Upsilon} v \leq 9\varepsilon \deg_{G_1} v +14 \leq \deg_{G_1} v/96\ell^2$  for $d$ sufficiently large and $\deg_{\Lambda} v\geq \deg_{G_3}v /2\ell\geq \deg_{G_1} v/8\ell$ for every $v$. Thus, $\max_{v}\left(\frac{\deg_{\Upsilon} v}{\deg_{\Lambda} v}\right)\leq 1/12\ell$.

\end{inproof}

If $\Lambda'$ is Eulerian, Theorem~\ref{thm:c-eulerian} implies that $\Lambda'$ has a non-conflicting Eulerian tour. If $\Lambda'$ is not Eulerian, it has two vertices of odd degree. We add a path between these two vertices which is not conflicting with any other path in $\Lambda'$ to obtain an Eulerian graph. (This can be done for instance by adding a wiggly edge with both indices $-1$ to the underlying graph of $\Lambda'$ and adding a path consisting of this edge.) Applying Theorem~\ref{thm:c-eulerian} to $\Lambda'$ with the additional path yields an Eulerian tour, which becomes an Eulerian trail after removing the additional path.

In either case, Observation~\ref{obs:eul-dec} implies that $\underline{\Lambda'}$ has an $\ell$-path-decomposition. This decomposition together with $\PP$ and $\LL$ yields an $\ell$-path-decomposition of $G_0$.


\end{proof}

\section{Proof of Theorem~\ref{thm:main}}\label{sec:paths}
\begin{proof}[Proof of Theorem~\ref{thm:main}]
The theorem trivially holds for $\ell=1$. Thus, we assume that $\ell\geq 2$ for the rest of the proof.
If as in Lemma~\ref{lem:core}, the graph $G$ is $k$-edge-connected, the result follows immediately. Assume that this is not the case. Let $K=3k$, where $k$ is as in Lemma~\ref{lem:core} for $\ell$, $\beta=1/1000$ and let $d=\max(D_1, D_2)$, where $D_1$ and $D_2$ are specified later in the proof.

Let $(H_0, G_0):= (\emptyset, G)$. We repeat the following {\em reducing} procedure, obtaining pairs of complex hypergraphs $(H_i,G_i)$ until $G_i$ is empty. Let $n$ be the number of iterations, i.e., $G_{n}$ is empty. 

If $G_i$ is $K$-edge-connected or has only one vertex, we let $H_{i+1}=G_i$ and $G_{i+1}=\emptyset$. Otherwise, by Lemma~\ref{lem:hycut}, there is a cut $(A,B)$ in $G_i$ of value (and order) less than $3K$ such that $G[A]$ is $K$-edge-connected or has only one vertex. 
 
Let $C_{i+1}$ be the set of edges and hyperedges containing at least one vertex in both $A$ and $B$. Let $H_{i+1}$ be the complex hypergraph obtained from $G_i[A]$ by adding a stub incident with a vertex $v$ for each incidence of an edge or hyperedge $e$ in $C_{i+1}$ with $v$. If $e$ is a wiggly edge or a hyperedge, we let the index of the corresponding stub be equal to the corresponding index of the edge or hyperedge. Otherwise, we specify the index of the stub later. We will specify lengths of the stubs later in the proof; for the purpose of the reducing procedure, the lengths are irrelevant (so we can set them arbitrarily). Let $S_{i+1}$ be the set of these stubs. Note that $|S_{i+1}|\leq 3K$ and $H_{i+1}$ does not contain any other stubs.
 
Let $G_{i+1}$ be the complex hypergraph with the vertex set $B$ constructed from $G_i$ as follows. We contract $A$ into a single vertex $a$, preserving multiplicity, and delete all the loops incident with $a$. We change all the ordinary edges incident with $a$ into wiggly edges with index $i+1$ at their endpoint in $B$. If there is a hyperedge with one vertex in $A$ and two vertices in $B$, we replace it by two wiggly edges from $a$ to the vertices in $B$. If there is a hyperedge with two vertices in $A$ and one vertex in $B$, we replace it with a single wiggly edge from $a$ to the vertex in $B$. For every such edge, we let the index at the endpoint in $B$ be the same as the index of the original hyperedge at that vertex. (We do not define the lengths and indices at $a$; they become irrelevant in the next step of the construction.) Since we contracted $A$ and $G_i$ was $3$-edge-connected, the resulting complex hypergraph is $3$-edge-connected. 

Then, we repeatedly replace pairs of edges incident with $a$ by wiggly edges with both endpoints in $B$, so that the connectivity between the vertices in $B$ does not change, until $a$ is an isolated vertex or has degree three. (This process corresponds to splitting off pairs of edges in the underlying graph.) This is possible by Corollary~\ref{cor:mader}. We let the index of each newly created edge at both endpoints be equal to the index of the edge which was replaced at the corresponding vertex.

If $a$ has degree three at the end of the process, we replace the three edges incident with $a$ by a hyperedge consisting of the three neighbors of $a$, letting the index of the hyperedge at each of its vertices be equal to the index of the edge which is being replaced. See Figure~\ref{fig:contr} for an example. Let $W_{i+1}$ be the set of wiggly edges and possibly the hyperedge arising from this operation. Again, we will specify lengths of the wiggly edges and the hyperedge in $W_{i+1}$ later.

Recall that the replacements do not change the connectivity between the vertices of $B$. Therefore, deleting the isolated vertex $a$ yields a $3$-edge-connected complex hypergraph on the vertex set $B$, which we denote $G_{i+1}$. Note that $\deg_{G_{i+1}}v=\deg_{G_i} v$ for every $v\in G_{i+1}$.  Moreover, by construction of $G_{i+1}$, every index in $G_{i+1}$ is in $[i+1]$ and $\iota(G_{i+1})\leq 3K$, since the index $i+1$ appears at most $3K$ times and the number of occurrences of any other index in $G_{i+1}$ is at most the number of its occurrences in $G_i$.  

\begin{figure}
\includegraphics[scale=1]{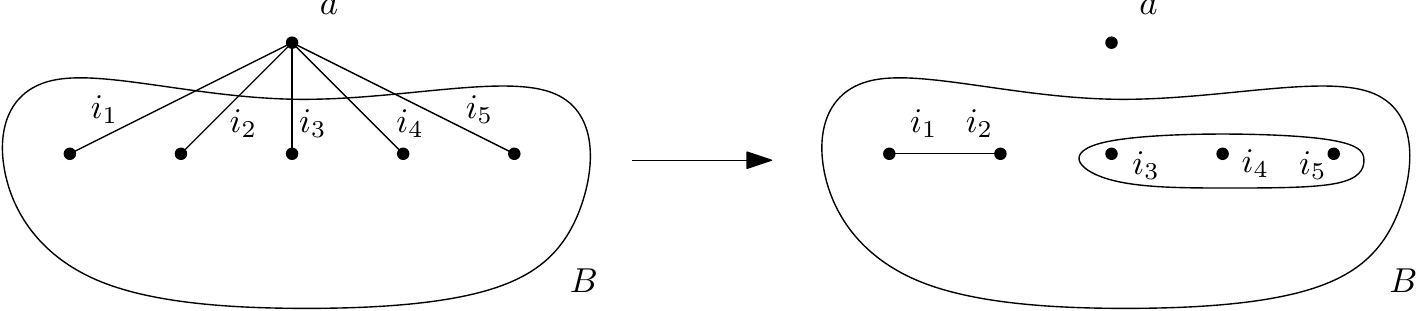}
\caption{Replacing edges between $a$ and $B$. Indices at vertices in $B$ "transfer" to the new edge or hyperedge.}
\label{fig:contr}
\end{figure}

Some of the wiggly edges and hyperedges of $W_i$ will be shrinked later in the proof. To keep the notation simple we update $W_i$ by replacing the shrinked wiggly edges and hyperedges by the wiggly edges and stubs arising from the shrinking.
 
Note that $V(G)=\dot{\bigcup}_{i\in [n]}V(H_i)$. Observe, that the following holds for every $i=1,\ldots, n$:

\begin{itemize}
\item $H_i$ is $K$-edge-connected or has only one vertex,
\item $H_i$ has minimum degree at least $d$ and stub degree at most $3K$
\item $G_i$ has minimum degree at least $d$ (minimum of the empty set is $\infty$) and has no stubs, and
\item $G_i$ is $3$-edge-connected, has only one vertex, or is empty.
\end{itemize}
 
\begin{figure}
\includegraphics[scale=1]{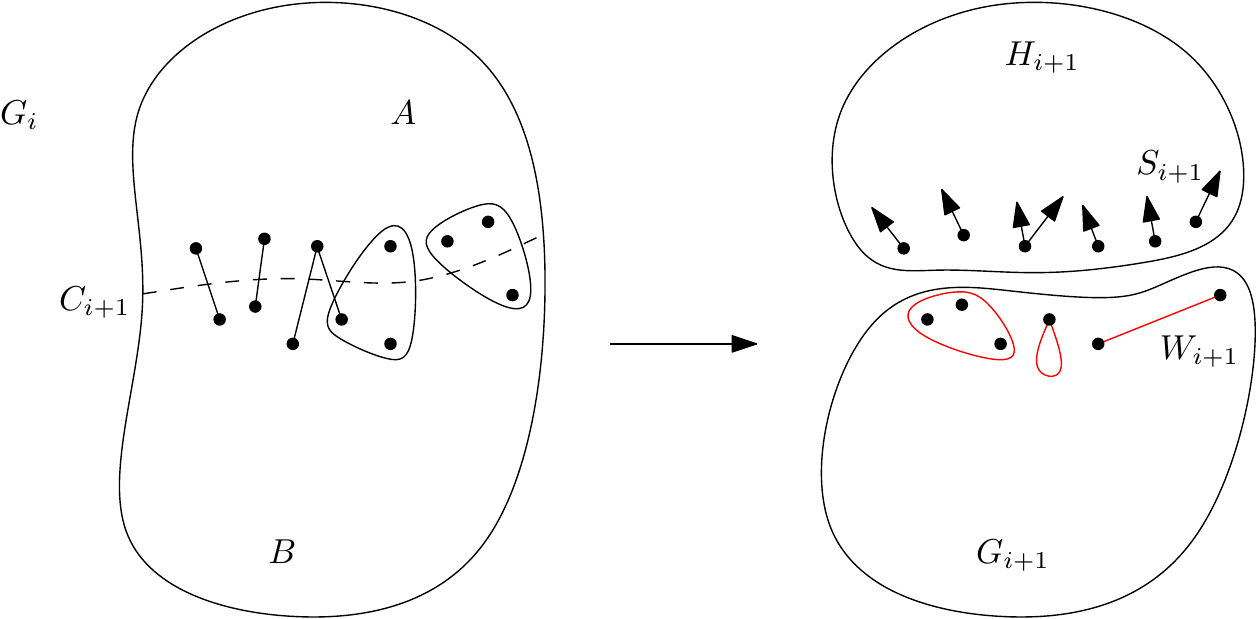}
\caption{The reducing procedure. (Stubs are depicted as arrows.)}
\label{fig:reduc}
\end{figure}

An example of the reducing procedure and the notation is depicted in the Figure~\ref{fig:reduc}.



Recall that $(V(H_i))_{i\in [n]}$ forms a partition of $V(G)$. Let $H=\dot{\bigcup}_{i\in [n]} H_i$ (thus, $V(H)=V(G)$). Let $V^{*}\subseteq V(H)$ be the set of vertices with $\hdeg_H v < \deg_H v/2$. Then, $\edeg_H v \geq \deg_H v/2 -\sdeg_H v\geq  d/2-3K$ for every $v\in V^*$.

We now assign lengths to the edges and hyperedges in the sets $W_i$. Note that the sets $W_i$ are disjoint, each of them contains at most one hyperedge and $|W_i|\leq 3K$. Also note that not all edges and hyperedges of the sets $W_i$ appear in $H$ (but we assign a length to all).

Let $\II\subseteq [n-1]$ be such that $W_i$ does not contain any hyperedge if and only if $i\in \II$. 
We select one wiggly edge $b_i$ to be {\em bad} in each set $W_i$, $i\in \II$, in such a way that every vertex $v\in V^*$ occurs in bad edges at most $4/5 \edeg_H v$ times.

\begin{claim}
There exists a set of bad edges $\BB=\{b_i\}_{i\in \II}$ such that $b_i\in W_i$ for every $i\in \II$ and every vertex $v\in V^{*}$ occurs  in bad edges at most $4/5\edeg_H v$ times. (Note that a vertex can appear up to two times in a wiggly edge.)

\end{claim}

\begin{inproof}

Let $N_v$ be the set of edges incident with a vertex $v\in V^{*}$ in $H$. We first observe that if at most $2/3$ of edges of $N_v$ are bad, $v$ occurs in bad edges at most $4/5\edeg_H v$ times. 

The vertex $v$ occurs at most twice and at least once in each edge of $N_v$. Thus, if at most $2/3$ of the edges are bad, the number $b_v$ of occurrences of $v$ in bad edges is at most $4|N_v|/3$. On the other hand, the number $g_v$ of occurrences of $v$ in the edges that are not bad is at least $|N_v|/3$. Thus, $b_v/g_v\leq 4$ and since $b_v+g_v=\edeg_{H} v$, $b_v\leq 4/5 \edeg_{H} v$. 

We show that there exists a set of bad edges $\BB$ such that at most $2/3$ of edges of $N_v$ are bad for every $v\in V^{*}$ using Lov\'asz Local Lemma. Assume that in each $W_i$, $i\in \II$, we independently select the bad edge uniformly at random. Since every $G_{i-1}$ is $3$-edge-connected, every $W_i$, $i\in \II$, contains at least two wiggly edges. Therefore, the probability of an edge of $W_i$, $i\in \II$, being bad is at most one half and these events are independent for edges from different $W_i$-s.

For $v\in V^{*}$, let $E_{v}$ be the event that more than $2/3$ of edges of $N_v$ are bad.  By our observation above, $\Pr(\bigwedge_{v\in V^{*}} \overline{E_v})>0$, implies the claim.

The event $E_{v}$ is independent of all except at most $3K\edeg_H v$ other events $E_{v'}$. (In particular, $E_v$ and $E_{v'}$ are independent unless both $v$ and $v'$ are incident with an edge of $W_i$, $i\in \II$
.)

Let $\pb$ be the set of edges $e_i$ such that $e_i$ is the only edge in $W_i\cap N_v$, $i\in \II$. Observe that at most half of the edges in $N_v\setminus \pb$ are bad (each such edge is either not in $W_i$ for any $i\in \II$ and therefore it cannot be bad, or there are at least two edges from the same $W_i$, out of which at most one is bad). Thus, if $|\pb|\leq |N_v|/3$, the probability of $E_v$ is $0$.

Assume it is not the case. Then, we show that at most $2/3$ of the edges of $\pb$ are bad (and thus, at most $2/3$ of the edges of $N_v$ are bad) with sufficiently small probability. The claim follows.

Since $|N_v|\geq \edeg_H v/2$, $|\pb|\geq |N_v|/3\geq \edeg_H v/6$. Note that the events "$e$ is bad" for $e\in\pb$ are independent.
Then, by Chernoff inequality, the probability that more than $2/3$ of edges in $\pb$ are bad is less than $c^{|\pb|/2}\leq c^{\edeg_H v/12}$, where $c< 0.98$.
Thus, at most $2/3$ of edges in $\pb$ are bad with probability less than $c^{\edeg_H v/12}$. 



It is possible to choose $D_1$ such that \[c^{\edeg_H v/12}\leq \frac{1}{\edeg_H v}(1-\frac{1}{d/2-3K})^{3K \edeg_H v}\] for every $v\in V^*$ if $d\geq D_1$ and $\edeg_H v> d/2-3K$. 
Then, for $x(E_v)=1/\edeg_H v$ for every $v\in V^*$, the asymmetric version of Lov\'asz Local Lemma (Lemma~\ref{LLL}) yields that $\Pr(\bigwedge_{v\in V^{*}} \overline{E_v})>0$, because \[\Pr(E_v)<\frac{1}{\edeg_H v}(1-\frac{1}{d/2-3K})^{3K \edeg_H v}\] for every $E_v$, $v\in V^{*}$.

\end{inproof}

Let $b_i$ be either the (only) bad wiggly edge or the hyperedge in $W_i$ for every $i\in [n-1]$. (Note that $W_n=\emptyset$ by the construction.) 

For each $W_i$, $i\in [n-1]$, we assign length one to all the wiggly edges in $W_i$ except $b_i$. It follows that every vertex $v$ has either $\hdeg_H v\geq\deg_H v/2$ or $\udeg_H v\geq\edeg_H v/5$.

We now assign lengths to $b_i$-s so that each $G_i$ is $\ell$-divisible. 

Note that for every $i\in[n-1]$, $W_i$ belongs to $G_i$ (in particular, $b_i$ is in $G_i$) and $G_i$ does not contain wiggly edges and hyperedges of $W_j$ for every $j>i$. Therefore, $b_1$ is the only wiggly edge or hyperedge in $G_1$ without assigned length. Let $\alpha(b_1)=-\alpha(G_1\setminus b_1) \mod \ell$. Then, $G_1$ is $\ell$-divisible. Inductively, we assign lengths to $b_i$, $i>1$, in an analogous way. 


We now construct an $\ell$-path decomposition of $G$. Since $G=G_0$ and it does not contain any wiggly edges and hyperedges, there is no shrinking of $G_0$ different from $G_0$ (but we consider $G_0$ to be its own (total)  shrinking). Thus, we can prove that $G$ has an $\ell$-path decomposition by showing that there is a total shrinking of $G_i$ that has an $\ell$-path decomposition for every $i\in \{0\}\cup [n]$. 
Assuming that a total shrinking $G'_i$ of $G_i$ has an $\ell$-path decomposition, we show that there is a total shrinking of $G_{i-1}$ that has an $\ell$-path decomposition. Since $G_n$ is the empty graph it has a trivial $\ell$-path decomposition (and it is its own total shrinking). 

By Observation~\ref{obs:abs-shrink}, if $G'_i$ has an $\ell$-path decomposition, there is an absolute shrinking $G^s_i$ of $G'_i$ (and also of $G_i$), that has an $\ell$-decomposition.

Now, the set $W_i$ in $G_i^s$ consists only of stubs. Each of these stubs corresponds to an edge or a hyperedge in  $C_i$. Recall that $H_i$ contains the set of stubs $S_i$ corresponding to incidences of edges and hyperedges of $C_i$ with vertices of $H_i$ in $G_{i-1}$. We have not specified  the lengths of these stubs yet. We assign the lengths in such a way that the sum of the lengths of the stubs in $S_i$ and $W_i$ corresponding to an edge or a hyperedge $e$ in $C_i$ is congruent to the length of $e$ modulo $\ell$ for every $e\in C_i$ (recall that we defined length of ordinary edges to be one). In case $e$ is a hyperedge with two vertices in $V(H_i)$, we assign  any lengths with the required sum to the corresponding stubs in $S_i$. Otherwise the length of a stub in $S_i$ is determined by the length of the corresponding stub in $W_i$ up to multiples of $\ell$.

Assuming that we have already constructed an $\ell$-path decomposition $\PP$ of $G_i^s$, we specify the indices of stubs in $S_i$ that correspond to ordinary edges in $C_i$. Let $s^e_G$ and $s^e_H$ be the stubs in $W_i$ and $S_i$, respectively, corresponding to an ordinary edge $e$ in $C_i$ and let $P_e\in \PP$ be the path containing $s^e_G$. Let $i$ be the index of $P_e\setminus s^e_G$ at the end where $s^e_G$ was attached. We let the index of $s^e_H$ be equal to $i$. 

Observe that $\iota(H_i)\leq 3K$ for every $i$ and all indices in $H_i$, $G_i$ are at most $i$. 

If all the edges in $C_i$ are wiggly edges or hyperedges, $G_i^s\cup H_i$ is a shrinking of $G_{i-1}$. In that case, it suffices to find an $\ell$-path decomposition of a total shrinking of $H_i$ to obtain an $\ell$-path decomposition of a total shrinking of $G_{i-1}$. 

In case $C_i$ contains some ordinary edges, the situation is more complex because ordinary edges cannot be shrinked. However, if there is an $\ell$-path decomposition $\QQ$ of a total shrinking $H^s_i$ of $H_i$ such that each path in $\QQ$ contains at most one stub in $S_i$, we are done. Let $e\in C_i$ be an ordinary edge and let $Q_e\in \QQ$ be the $\ell$-path containing $s^e_H$. Then, $(P_e\setminus s^e_G)\ct e\ct (Q_e\setminus s^e_H)$ is an $\ell$-path. Thus, the graph obtained from $H_i^s\cup G_i^s$ by replacing stubs in $W_i$ and $S_i$ corresponding to ordinary edges in $C_i$ by these edges has an $\ell$-path decomposition.

We construct a desired $\ell$-path decomposition of a total shrinking of $H_i$ as follows.
First, note that $H_i$ is $\ell$-divisible; by the choice of lengths of stubs in $S_i$, we have $\alpha(H_i)+\alpha(G_i^s)=\alpha(G_{i-1})\mod \ell$. Since both $G_{i-1}$ and $G_i^s$ are $\ell$-divisible, $H_i$ is $\ell$-divisible as well.

By applying Lemma~\ref{lem:shrink} to $H_i\setminus S_i$, we obtain a total shrinking $H'_i$ of $H_i\setminus S_i$ that is $K/3$-edge-connected, $\edeg_{H'_{i}} v\geq\edeg_{H_{i}\setminus S_i} v/100$ and $\udeg_{H'_{i}} v\geq\edeg_{H'_{i}} v/1000$ for every $v$, and all except at most $S$ stubs at each vertex in $H_i'$ form balanced stub-pairs where $S$ is as in Lemma~\ref{lem:shrink}.

Let $H''_{i}$ be the complex graph obtained from $H'_{i}$ by removing all balanced stub-pairs. From Observation~\ref{obs:balanced-stub}, it follows that if $H''_{i}\cup S_i$ has an $\ell$-path decomposition $\PP$, $H_i'$ has an $\ell$-path decomposition $\PP\cup \SSS$ where $\SSS$ is a set of $\ell$-paths formed by the balanced stub-pairs.
Note that $H''_{i}\cup S_i$ has at most $S+3K$ stubs at each vertex.

Moreover, by construction, $H''_{i}\cup S_i$ satisfies the assumptions of Lemma~\ref{lem:core} if $d$ is sufficiently large. To ensure this, we choose $D_2$ to be as $d$ in Lemma~\ref{lem:core} for $\ell$, $\beta=1/1000$, $k=K/3$, $r=3K$ and $s=3K+S$. Then, $H''_{i}\cup S_i$ has an $\ell$-path decomposition. As discussed above, it follows that $G$ has an $\ell$-path decomposition.
\end{proof}


\section*{Acknowledgements}
We would like to thank to Julien Bensmail, Ararat Harutyunyan and Tien-Nam Le 
for helpful discussions.

The work was done while the first author was a postdoc at Laboratoire d’Informatique du Parall\'elisme,
\'Ecole Normale Sup\'erieure de Lyon,
69364 Lyon Cedex 07, France.
\bibliography{biblio} 

\begin{thebibliography}{10}

\bibitem{bib:barat-thomassen}
J.~Bar{\'a}t and C.~Thomassen.
\newblock Claw-decompositions and {T}utte-orientations.
\newblock {\em Journal of Graph Theory}, 52(2):135--146, 2006.

\bibitem{bib:BHLMT}
J.~Bensmail, A.~Harutyunyan, T.-N. Le, M.~Merker, and S.~Thomass{\'e}.
\newblock A proof of the {B}ar\'at-{T}homassen conjecture.
\newblock {\em J. Combin. Theory Ser. B}, 124:39--55, 2017.

\bibitem{bib:orig-paths}
J.~Bensmail, A.~Harutyunyan, T.-N. Le, and S.~Thomass{\'e}.
\newblock Edge-partitioning a graph into paths: Beyond the
  {B}ar{\'a}t-{T}homassen conjecture.
\newblock {\em Combinatorica}, 39(2):239--263, Apr 2019.

\bibitem{bib:lll}
P.~Erd\H{o}s and L.~Lov{\'a}sz.
\newblock Problems and results on 3-chromatic hypergraphs and some related
  questions.
\newblock In R.~R. A.~Hajnal and V.~S\'os, editors, {\em Infinite and finite
  sets}, pages 609--627. North-Holland, II, 1975.

\bibitem{bib:nwh}
A.~Frank, T.~Király, and M.~Kriesell.
\newblock On decomposing a hypergraph into $k$ connected sub-hypergraphs.
\newblock {\em Discrete Applied Mathematics}, 131(2):373 -- 383, 2003.
\newblock Submodularity.

\bibitem{bib:jackson}
B.~Jackson.
\newblock On circuit covers, circuit decompositions and euler tours of graphs.
\newblock {\em Surveys in Combinatorics, 1993 (Keele)}, pages 191--210, 1993.

\bibitem{bib:our-trees}
T.~Klimo\v{s}ov\'{a} and S.~Thomass{\'e}.
\newblock Edge-decomposing graphs into coprime forests.
\newblock {\em arXiv preprint arXiv:1803.03704}, 2018.

\bibitem{bib:lovasz}
L.~Lov{\'a}sz.
\newblock A generalization of {K}{\"o}nig's theorem.
\newblock {\em Acta Mathematica Hungarica}, 21(3-4):443--446, 1970.

\bibitem{bib:LOVASZ2013587}
L.~M. Lovász, C.~Thomassen, Y.~Wu, and C.-Q. Zhang.
\newblock Nowhere-zero 3-flows and modulo k-orientations.
\newblock {\em Journal of Combinatorial Theory, Series B}, 103(5):587 -- 598,
  2013.

\bibitem{bib:mader}
W.~Mader.
\newblock A reduction method for edge-connectivity in graphs.
\newblock In B.~Bollobás, editor, {\em Advances in Graph Theory}, volume~3 of
  {\em Annals of Discrete Mathematics}, pages 145 -- 164. Elsevier, 1978.

\bibitem{bib:entropy}
R.~A. Moser and G.~Tardos.
\newblock A constructive proof of the general {L}ov{\'a}sz local lemma.
\newblock {\em Journal of the ACM (JACM)}, 57(2):11, 2010.

\bibitem{bib:thomassen-3-flow}
C.~Thomassen.
\newblock The weak 3-flow conjecture and the weak circular flow conjecture.
\newblock {\em Journal of Combinatorial Theory, Series B}, 102(2):521 -- 529,
  2012.

\bibitem{bib:thomassen-paths}
C.~Thomassen.
\newblock Decomposing graphs into paths of fixed length.
\newblock {\em Combinatorica}, 33(1):97--123, 2013.

\end{thebibliography}
\bibliographystyle{abbrv}

\end{document}